\documentclass[12pt]{amsart}
\usepackage{amsmath,amsthm,amssymb,amsfonts}
\usepackage{mathtools}
\usepackage[margin=2cm]{geometry}
\usepackage{enumerate}
\usepackage{graphicx}
\usepackage{subfig}
\usepackage{latexsym}
\usepackage{tikz}
\usetikzlibrary{arrows,arrows.meta}
\usepackage{hyperref}

\def\ZZ{\mathbb{Z}}

\newcommand{\cC}{\mathcal{C}}
\newcommand{\cG}{\mathcal{G}}
\newcommand{\cQ}{\mathcal{Q}}
\newcommand{\cZ}{\mathcal{Z}}

\def\opn#1#2{\def#1{\operatorname{#2}}}
\opn\chara{char} \opn\pd{pd} \opn\depth{depth}
\opn\reg{reg} \opn\lcm{lcm} \opn\supp{supp} \opn\Tor{Tor}
\opn\Im{Im} \opn\st{st} \opn\core{core} \opn\sd{sd}
\opn\Nb{Nb}
\let\kappa=\varkappa
\let\iso=\cong

\newcommand{\Bp}{\textup{(B$_p$)}}
\newcommand{\Btwo}{\textup{(B$_2$)}}

\newtheorem{Theorem}{Theorem}[section]
\newtheorem{Lemma}[Theorem]{Lemma}
\newtheorem{Corollary}[Theorem]{Corollary}
\newtheorem{Proposition}[Theorem]{Proposition}
\theoremstyle{definition}
\newtheorem{Definition}[Theorem]{Definition}
\newtheorem{Construction}[Theorem]{Construction}
\newtheorem{Remark}[Theorem]{Remark}
\newtheorem{Example}[Theorem]{Example}
\newtheorem{Discussion}[Theorem]{Discussion}
\newtheorem{Question}[Theorem]{Question}

\begin{document}

\title[On the Linearity of Squarefree Powers of Edge Ideals]{On the Linearity of Squarefree Powers of Edge Ideals}

\author[F. Navarra]{Francesco Navarra}
\address[Francesco Navarra]{Sabanci University, Faculty of Engineering and Natural Sciences, Orta Mahalle, Tuzla 34956, Istanbul, Turkey}
\email{francesco.navarra@sabanciuniv.edu}

\author[A. A. Qureshi]{Ayesha Asloob Qureshi}
\address[Ayesha Asloob Qureshi]{Sabanci University, Faculty of Engineering and Natural Sciences, Orta Mahalle, Tuzla 34956, Istanbul, Turkey}
\email{aqureshi@sabanciuniv.edu, ayesha.asloob@sabanciuniv.edu}

\author[N. Terai]{Naoki Terai}
\address[Naoki Terai]{Department of Mathematics, Okayama University, 3-1-1 Tsushima-naka, Kita-ku, Okayama 700-8530, Japan}
\email{terai@okayama-u.ac.jp}



\begin{abstract}
Let $G$ be a graph and $I(G)$ its edge ideal. The $p$-th squarefree power $I(G)^{[p]}$ is the monomial ideal generated by the squarefree monomials corresponding to the matchings of size $p$ of $G$. A central highlight of our work is the complete solution of a problem posed by Bigdeli, Herzog and Zaare-Nahandi in 2017, which asked to characterize when squarefree powers of edge ideals have the first syzygy module generated by linear forms. Moreover, for a $1$-dimensional flag simplicial complex $\Delta$ and its Stanley--Reisner ideal $I_\Delta$, which is the edge ideal of the complement graph of $\Delta$, we prove that $I_\Delta^{[p]}$ has a linear resolution if and only if it is linearly related and the $(2p-1)$-st homology of the neighbourhood complex of $\Delta$ vanishes. As a consequence, we describe the shape of the Betti table of $I_\Delta^{[p]}$, and we give a combinatorial characterization of when $I_\Delta^{[p]}$ has a linear resolution, provided that the $2p$-core of $\Delta$ has maximum degree at most $2p$. Finally, we show that no such characterization exists in general: the linearity of $I_\Delta^{[2]}$ may depend on the characteristic of the field, and it cannot be described by finitely many forbidden induced subgraphs.
\end{abstract}

\maketitle

\section*{Introduction}

The relation between the homological properties of squarefree monomial ideals and the combinatorics and topology of the associated simplicial complexes and graphs has been a central theme in Combinatorial Commutative Algebra over the last decades; see \cite{BH, kozlov2007combinatorial, miller2004combinatorial, Villarreal}. Edge ideals of graphs, introduced by Villarreal \cite{Villarreal1990}, are among the most studied examples. A classical result in this direction is Fr\"oberg's theorem \cite{Froberg1990}: the edge ideal $I(G)$ of a graph $G$ has a linear resolution if and only if the complement of $G$ is chordal. In particular, for edge ideals the existence of a linear resolution does not depend on the base field, unlike what happens for arbitrary squarefree monomial ideals: a classical example is provided by triangulations of the real projective plane.

In this paper we study the minimal graded free resolutions of squarefree powers of edge ideals. Given a monomial ideal $I$, its $k$-th squarefree power $I^{[k]}$ is the ideal generated by the squarefree monomials among the minimal generators of $I^k$. Squarefree powers of edge ideals were systematically introduced in \cite{BHZ}, although they had implicitly appeared earlier, for instance in the proof of \cite[Proposition~2.10]{JZ}, and they have since been investigated by many authors; see \cite{CFL, das2026regularitysquarefreepowersedge, EHHM2, EH, Kamalesh1, KAMBERI2026240, youngcomplexes, SEYEDFAKHARI2024107488} and the references therein. Two features make them interesting. On the algebraic side, the multigraded minimal free resolution of $I^{[k]}$ is a subcomplex of that of $I^k$ \cite[Lemma~4.4]{HerzogHibiZheng}; hence $\reg(I^{[k]})\leq\reg(I^k)$, and if $I^{[k]}$ does not have a linear resolution, neither does $I^k$. On the combinatorial side, if $I(\Delta)$ is the facet ideal of a simplicial complex $\Delta$, the minimal generators of $I(\Delta)^{[k]}$ correspond to the matchings of $\Delta$ of size $k$, so that $I(\Delta)^{[k]}\neq0$ if and only if $k\leq\nu(\Delta)$, the matching number of $\Delta$.

Bigdeli, Herzog and Zaare-Nahandi proved that the highest nonvanishing squarefree power of an edge ideal has linear quotients, hence a linear resolution \cite[Theorem~5.1]{BHZ}, and conjectured that $I(G)^{[p]}$ is linearly related if and only if $p\geq\nu_0(G)$, the restricted matching number of $G$ \cite[Conjecture~5.3]{BHZ}. A counterexample was later found in \cite[page~12]{EHHM}, which leaves open the following problem.

\medskip

\noindent\textbf{Problem.} Characterize the graphs $G$ for which $I(G)^{[p]}$ is linearly related.

\medskip

Our first main result gives a complete answer.

\begin{Theorem}[Theorem~\ref{Thm:linrelated1}]\label{Thm intro: lin rel}
Let $G$ be a graph and let $2\leq p\leq\nu(G)$. Then $I(G)^{[p]}$ is linearly related if and only if $G$ has no induced subgraph $H$ on $2p+2$ vertices such that $H$ is disconnected and $\nu(H)=p+1$.
\end{Theorem}

In particular, if $G$ has such an induced subgraph, then the ordinary power $I(G)^p$ does not have a linear resolution. The proof combines the connectivity criterion of \cite[Corollary~2.2]{BHZ} with the description of first syzygies through the lcm-lattice due to Gasharov, Peeva and Welker \cite{Gasharov1999}.

Motivated by these findings, we decided to push our boundaries beyond the linearity of the first syzygy module. The problem of characterizing when squarefree powers of edge ideals admit a linear resolution is significantly more subtle. We address it for graphs whose complement is triangle-free. Equivalently, we consider the Stanley--Reisner ideal $I_\Delta=I(\overline{\Delta})$ of a $1$-dimensional flag simplicial complex $\Delta$, that is, of a triangle-free graph, and its squarefree powers $I_\Delta^{[p]}$, $2\leq p\leq\nu(\overline{\Delta})$. In this setting Theorem~\ref{Thm intro: lin rel} becomes the following statement, in which the failure of linear relatedness is also quantified.

\begin{Theorem}[Corollary~\ref{coro: lin related} and Theorem~\ref{Thm: betti 2,2p+2}]\label{Thm intro: Delta lin rel}
Let $\Delta$ be a $1$-dimensional flag simplicial complex and let $2\leq p\leq\nu(\overline{\Delta})$. Then $I_\Delta^{[p]}$ is linearly related if and only if $\Delta$ has no induced subgraph isomorphic to $K_{2+i,\,2p-i}$ for an even integer $0\leq i\leq p$. Moreover, $\beta_{1,2p+2}(I_\Delta^{[p]})$ equals the number of such induced subgraphs of $\Delta$.
\end{Theorem}

To go beyond the first syzygies we use the \emph{neighbourhood complex} $\Nb(\Delta)$ of $\Delta$, the simplicial complex whose faces are the sets of vertices having a common neighbour. It was introduced by Lov\'asz in his proof of Kneser's conjecture \cite{Lovasz1978}. Our second main result shows that, once the first syzygies are linear, the linearity of the whole resolution depends only on one homology group of this complex.

\begin{Theorem}[Theorems~\ref{Thm: main criterion}
]\label{Thm intro: criterion}
Let $\Delta$ be a $1$-dimensional flag simplicial complex and let $2\leq p\leq\nu(\overline{\Delta})$. Then $I_\Delta^{[p]}$ has a linear resolution over a field $K$ if and only if $I_\Delta^{[p]}$ is linearly related and $\widetilde{H}_{2p-1}(\Nb(\Delta);K)=0$. 
\end{Theorem}

The proof relies on a dimension-shifting argument in a double complex built from the complete bipartite subgraphs of $\Delta$, which identifies the top-dimensional cycles of the Stanley--Reisner complex of $I_\Delta^{[p]}$ with the $(2p-1)$-cycles of $\Nb(\Delta)$. For the crown graph $Cr(2p+1)$, obtained from $K_{2p+1,2p+1}$ by removing a perfect matching, the neighbourhood complex consists of two copies of the boundary of a $2p$-simplex, and Theorem~\ref{Thm intro: criterion} shows that an induced crown always prevents $I_\Delta^{[p]}$ from having a linear resolution. For a comparison with a lemma of Kahle on neighbourhood complexes, see Remark~\ref{Rem: Kahle}.

Several characteristic-free consequences follow. Recall that the $k$-core of a graph is its largest induced subgraph of minimum degree at least $k$.

\begin{Theorem}[Theorems~\ref{Thm: core degree} and~\ref{Thm: first Betti}, Corollaries~\ref{Coro: planar} and~\ref{Coro: lin res for p then lin res for p+1}]\label{Thm intro: consequences}
Let $\Delta$ be a $1$-dimensional flag simplicial complex on $n$ vertices and let $2\leq p\leq\nu(\overline{\Delta})$. Then the following hold over every field.
\begin{enumerate}
\item[(a)] If every vertex of the $2p$-core of $\Delta$ has degree at most $2p$ in the core (for instance, if $\Delta$ has maximum degree at most $2p$), then $I_\Delta^{[p]}$ has a linear resolution if and only if $\Delta$ has no induced subgraph isomorphic to $K_{2+i,\,2p-i}$, for an even integer $0\leq i\leq p$, or to $Cr(2p+1)$.
\item[(b)] If $\Delta$ is planar, then $I_\Delta^{[p]}$ has a linear resolution if and only if it is linearly related.
\item[(c)] If $I_\Delta^{[p]}$ is linearly related, then $\beta_{j,j+2p+1}(I_\Delta^{[p]})=0$ for $j\leq2p$, and $\beta_{2p+1,4p+2}(I_\Delta^{[p]})$ equals $p$ times the number of induced subgraphs of $\Delta$ isomorphic to $Cr(2p+1)$.
\item[(d)] If $I_\Delta^{[p]}$ is linearly related and $p<\nu(\overline{\Delta})$, then $I_\Delta^{[p+1]}$ has a linear resolution. In particular, if $I_\Delta^{[p]}$ has a linear resolution, then so has $I_\Delta^{[q]}$ for every $p\leq q\leq\nu(\overline{\Delta})$.
\end{enumerate}
\end{Theorem}

One might expect, in analogy with Fr\"oberg's theorem and with Theorem~\ref{Thm intro: Delta lin rel}, that the assumption on the core in (a) can be dropped. This is not the case. In fact, if $\Delta$ is the vertex--facet incidence graph of the hypersimplex $\{t\in[0,1]^5:t_1+\cdots+t_5=2\}$, then $\Delta$ has $20$ vertices and contains no induced $K_{2,4}$ and no induced $Cr(5)$, but $I_\Delta^{[2]}$ does not have a linear resolution over any field (Example~\ref{Exa: hypersimplex}); here the $4$-core is the whole graph, and it has vertices of degree $5$ and $6$.

Moreover, no description in terms of the graph alone can exist. To see this, we attach to every simplicial complex $M$ of dimension at most $3$ in which each triangle lies in at most two facets a bipartite graph $\Delta_M$, whose vertices are the vertices, the edges and the facets of $M$, and we show that $\Nb(\Delta_M)$ has the homology of two copies of $M$. Consequently, in Corollary~\ref{Coro: Delta_M linear}, we get that $I_{\Delta_M}^{[2]}$ is always linearly related, and it has a linear resolution over $K$ if and only if $H_3(M;K)=0$. For the boundary $M$ of the $4$-simplex, $\Delta_M$ is the graph of the hypersimplex above, and $H_3(M;K)=K$ for every field $K$. Replacing $M$ by the suspension of the real projective plane, we make $H_3(M;K)$ depend on the characteristic.

\begin{Theorem}[Theorem~\ref{Thm: characteristic} and Corollary~\ref{Coro: no characteristic free}]\label{Thm intro: characteristic}
Let $M_0$ be the suspension of the six-vertex triangulation of the real projective plane. Then $\Delta_{M_0}$ is a graph on $55$ vertices, and $I_{\Delta_{M_0}}^{[2]}$ has a linear resolution over $K$ if and only if $\chara K\neq2$. In particular, there is no class $\mathcal{P}$ of finite graphs, independent of $K$, such that, for every field $K$ and every $1$-dimensional flag simplicial complex $\Delta$ with $\nu(\overline{\Delta})\geq2$, the ideal $I_\Delta^{[2]}$ has a linear resolution over $K$ if and only if $\Delta\in\mathcal{P}$.
\end{Theorem}

We say that $\Delta$ is a \emph{minimal forbidden graph} over $K$ if $I_\Delta^{[2]}$ does not have a linear resolution over $K$, while $I_{\Delta_W}^{[2]}$ has a linear resolution or is zero for every proper induced subgraph $\Delta_W$. By Theorem~\ref{Thm intro: Delta lin rel}, the only minimal forbidden graph for which $I_\Delta^{[2]}$ is not linearly related is $K_{2,4}$, and $Cr(5)$ is a minimal forbidden graph over every field. Nevertheless, the list is infinite.

\begin{Theorem}[Theorem~\ref{Thm: minimal forbidden} and Corollary~\ref{Coro: infinitely many}]\label{Thm intro: minimal forbidden}
Let $M$ be a strongly connected closed $3$-pseudomanifold with $H_3(M;K)\neq0$; for instance, let $M$ be any triangulated $3$-sphere. Then $\Delta_M$ is a minimal forbidden graph over $K$. Consequently, for every field $K$ there are infinitely many minimal forbidden graphs over $K$, and the class of graphs $\Delta$ for which $I_\Delta^{[2]}$ has a linear resolution over $K$ cannot be described by finitely many forbidden induced subgraphs.
\end{Theorem}

In particular, the graph $\Delta_{M_0}$ of Theorem~\ref{Thm intro: characteristic} is a minimal forbidden graph over $\mathbb{F}_2$, although $I_{\Delta_{M_0}}^{[2]}$ has a linear resolution over $\mathbb{Q}$.

Characteristic-dependent Betti numbers of edge ideals and of their powers have been studied before. Katzman \cite{Katzman2006} showed that the Betti numbers of edge ideals of graphs with at most $10$ vertices do not depend on the field, and found the first graphs, on $11$ vertices, for which they do; see also Dalili and Kummini \cite{DaliliKummini2014}. Bolognini, Macchia, Strazzanti and Welker \cite{BMSW2022} constructed monomial ideals, including edge ideals, all of whose ordinary powers have characteristic-dependent Betti numbers. In these works the characteristic affects Betti numbers. In our setting the characteristic affects the linearity of the resolution itself, for a class of ideals whose first syzygies do not depend on the field; moreover, the construction $M\mapsto\Delta_M$ shows that the characteristic enters only through the $3$-dimensional homology of $M$.

We also show that linearity over some field implies linearity in characteristic zero, and that for a given graph only finitely many characteristics are exceptional (Proposition~\ref{Prop: exceptional characteristics}).

The paper is organized as follows. Section~\ref{Sec1: Preliminaries} collects the necessary background. In Section~\ref{sec: 1-dim flag simplicial complexes} we study the Stanley--Reisner complex $\Delta^{[p]}$ of $I_\Delta^{[p]}$: we show that $\dim\Delta^{[p]}\in\{2p-2,2p-1\}$, describe its $(2p-1)$-dimensional faces in terms of induced complete bipartite subgraphs, characterize when it is pure and when it is Cohen--Macaulay, and compute the projective dimension of $I_\Delta^{[p]}$. Section~\ref{Sec: linearly related} contains the proofs of Theorems~\ref{Thm intro: lin rel} and~\ref{Thm intro: Delta lin rel}. In Section~\ref{sec: linearity simpl complex} we prove Theorems~\ref{Thm intro: criterion} and~\ref{Thm intro: consequences}, and in Section~\ref{Sec: characteristic} we construct the graphs $\Delta_M$ and prove Theorems~\ref{Thm intro: characteristic} and~\ref{Thm intro: minimal forbidden}. Finally, Section~\ref{Sec: final} describes the possible shapes of the Betti table of $I_\Delta^{[p]}$ and collects some open questions.

\section{Preliminaries}\label{Sec1: Preliminaries}

We fix the notation and recall the classical facts used throughout the paper.

\subsection{Simplicial complexes and graphs}\label{Subsec1}
A \textit{simplicial complex} $\Delta$ on the vertex set $[n]:=\{1,\dots,n\}$ is a nonempty family of subsets of $[n]$ such that $F\in\Delta$ and $F'\subseteq F$ imply $F'\in\Delta$. Its elements are the \textit{faces} of $\Delta$; the \textit{dimension} of a face $F$ is $\dim(F)=|F|-1$, and $\dim\Delta$ is the maximum dimension of a face. Faces of dimension $0$ and $1$ are called \textit{vertices} and \textit{edges}, the maximal faces are the \textit{facets}, and $\mathcal{F}(\Delta)$ denotes the set of facets. For $W\subseteq[n]$, the \textit{induced subcomplex} $\Delta_W$ consists of the faces of $\Delta$ contained in $W$. For subsets $F_1,\dots,F_m$ of $[n]$ we write $\langle F_1,\dots,F_m\rangle$ for the simplicial complex whose faces are the subsets of some $F_i$; a complex of the form $\langle F\rangle$ is a \textit{simplex}. The complex $\Delta$ is \textit{pure} if all its facets have the same dimension, and its $i$-th \textit{skeleton} $\Delta^{(i)}$ consists of the faces of dimension at most $i$. A pure complex of dimension $d-1$ is \textit{connected in codimension one} if any two facets $F,G$ can be joined by a sequence of facets $F=F_0,F_1,\dots,F_q=G$ with $|F_i\cap F_{i+1}|=d-1$ for all $i$.

A simplicial complex of dimension at most one is a \textit{graph}, and the notions above are then read in the language of graph theory; for instance, $\Delta_W$ is the \textit{induced subgraph} of $\Delta$ on $W$. The \textit{complement} of a graph $\Delta$ on $[n]$ is the graph $\overline{\Delta}$ on $[n]$ whose edges are the $2$-subsets of $[n]$ that are not edges of $\Delta$. A \textit{path} of length $r-1$ is a sequence of pairwise distinct vertices $v_1,\dots,v_r$ with $\{v_i,v_{i+1}\}\in\Delta$ for all $i$; if moreover $r\geq3$ and $\{v_r,v_1\}\in\Delta$, the sequence is a \textit{cycle} of length $r$, and a cycle of length three is a \textit{triangle}. A vertex is \textit{isolated} if it lies in no edge, and a set of vertices is \textit{independent} if it contains no edge. A graph is \textit{bipartite} if its vertex set is the disjoint union of two independent sets, equivalently, if it has no cycle of odd length \cite[Lemma~9.1.1]{HHmonomialideals}. The \textit{complete bipartite graph} with parts of sizes $m$ and $n$ is denoted by $K_{m,n}$, and the \textit{crown graph} $Cr(n)$ is the bipartite graph with vertex set $\{u_1,\dots,u_n\}\sqcup\{v_1,\dots,v_n\}$ and edges $\{u_i,v_j\}$ for all $i\neq j$; see Figure~\ref{fig:bipartite + crown}.

\begin{figure}[h]
    \centering
    \subfloat[]{\includegraphics[scale=0.5]{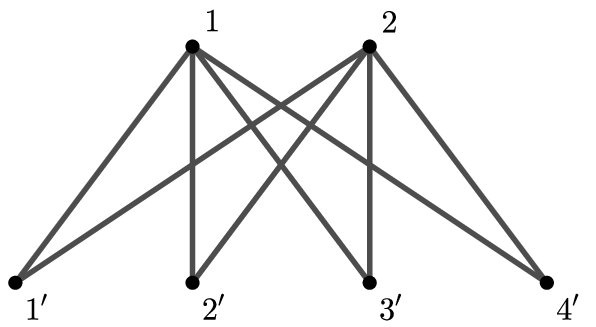}}\qquad
    \subfloat[]{\includegraphics[scale=0.5]{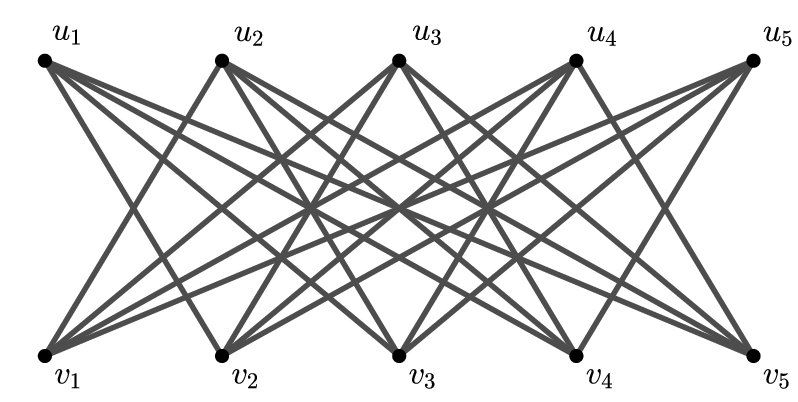}}
    \caption{The complete bipartite graph $K_{2,4}$ and the crown graph $Cr(5)$.}
    \label{fig:bipartite + crown}
\end{figure}

\subsection{Stanley--Reisner ideals, Betti numbers and squarefree powers}\label{SubSec2}
Throughout, $K$ is an arbitrary field and $S=K[x_1,\dots,x_n]$ is standard graded. For $F\subseteq[n]$ we put $x_F=\prod_{i\in F}x_i$. The \textit{facet ideal} of a simplicial complex $\Delta$ is $I(\Delta)=(x_F:F\in\mathcal{F}(\Delta))$; if $\Delta$ is a graph, $I(\Delta)$ is its \textit{edge ideal}. Denoting by $\mathcal{N}(\Delta)$ the set of minimal non-faces of $\Delta$, the \textit{Stanley--Reisner ideal} of $\Delta$ is $I_\Delta=(x_F:F\in\mathcal{N}(\Delta))$ and $K[\Delta]=S/I_\Delta$ is its \textit{Stanley--Reisner ring}. One has $\dim K[\Delta]=\dim\Delta+1=\max\{|F|:x_F\notin I_\Delta\}$ \cite[Corollary~6.3.5]{Villarreal}. The complex $\Delta$ is \textit{Cohen--Macaulay} over $K$ if $K[\Delta]$ is, and every Cohen--Macaulay complex is connected in codimension one \cite[Lemma~9.1.12]{HHmonomialideals}. The complex $\Delta$ is \textit{flag} if $I_\Delta$ is generated in degree two. A $1$-dimensional flag complex is a triangle-free graph, and conversely; in this case $I_\Delta=I(\overline{\Delta})$, that is, the Stanley--Reisner ideal of $\Delta$ is the edge ideal of the complement graph.

Let $I\subset S$ be a homogeneous ideal generated in degree $t$, with minimal graded free resolution
\[
0\to\bigoplus_{j}S(-j)^{\beta_{\ell,j}}\to\cdots\to\bigoplus_{j}S(-j)^{\beta_{1,j}}\to\bigoplus_{j}S(-j)^{\beta_{0,j}}\to I\to0.
\]
The integers $\beta_{i,j}=\beta_{i,j}(I)$ are the \textit{graded Betti numbers} of $I$, $\ell=\pd(I)$ is its \textit{projective dimension}, and $\reg(I)=\max\{j-i:\beta_{i,j}(I)\neq0\}$ is its \textit{regularity}; recall that $\reg(S/I)=\reg(I)-1$ and $\pd(S/I)=\pd(I)+1$. The ideal $I$ has a \textit{linear resolution} if $\beta_{i,j}(I)=0$ whenever $j\neq i+t$, and it is \textit{linearly related} if $\beta_{1,j}(I)=0$ whenever $j\neq t+1$.

For a simplicial complex $\Delta$ with a total order on its vertices, $C_i(\Delta)$ is the $K$-vector space with basis the oriented faces $[v_1,\dots,v_{i+1}]$, $v_1<\cdots<v_{i+1}$, and $C_{-1}(\Delta)=K$. For vertices $w_1,\dots,w_k$ listed in an arbitrary order we put $[w_1,\dots,w_k]=\operatorname{sgn}(\sigma)[w_{\sigma(1)},\dots,w_{\sigma(k)}]$, where $\sigma$ sorts them increasingly. The augmented chain complex $(C_\bullet(\Delta),\delta_\bullet)$ has differentials
\[
\delta_i([v_1,\dots,v_{i+1}])=\sum_{j=1}^{i+1}(-1)^{j+1}[v_1,\dots,\widehat{v_j},\dots,v_{i+1}],
\]
so that $\delta_0(v)=1$, and $\widetilde{H}_i(\Delta;K)=\ker\delta_i/\Im\delta_{i+1}$ is the $i$-th reduced homology of $\Delta$.
By convention $\widetilde{H}_{-1}(\{\emptyset\};K)=K$, and $\dim_K\widetilde{H}_0(\Delta;K)$ is the number of connected components of $\Delta$ minus one. We denote by $H_i(\Delta;R)$ the $i$-th simplicial homology of $\Delta$ with coefficients in a commutative ring $R$, and we recall that $\widetilde{H}_i(\Delta;R)=H_i(\Delta;R)$ for every $i\geq1$.
The graded Betti numbers of $I_\Delta$ are given by \textit{Hochster's formula}
\[
\beta_{i,j}(I_\Delta)=\sum_{W\subseteq[n],\ |W|=j}\dim_K\widetilde{H}_{j-i-2}(\Delta_W;K).
\]
In particular $\reg K[\Delta]\leq\dim K[\Delta]$.

Let $I$ be a squarefree monomial ideal. Its $k$-th \textit{squarefree power} $I^{[k]}$ is the ideal generated by the squarefree monomials among the minimal generators of $I^k$ \cite{BHZ}. For instance, if $I=(x,y,z)$, then $I^{[2]}=(xy,xz,yz)$, $I^{[3]}=(xyz)$ and $I^{[k]}=(0)$ for $k\geq4$. A \textit{matching} of a simplicial complex $\Delta$ is a set of pairwise disjoint facets, a $k$-\textit{matching} is a matching of cardinality $k$, and the \textit{matching number} $\nu(\Delta)$ is the largest cardinality of a matching. If $I=I(\Delta)$, then
\[
I^{[k]}=(x_{F_1}\cdots x_{F_k}:\{F_1,\dots,F_k\}\text{ is a $k$-matching of }\Delta),
\]
so that $I^{[k]}\neq0$ if and only if $k\leq\nu(\Delta)$. In particular, for a $1$-dimensional flag complex $\Delta$ we have $I_\Delta^{[p]}=I(\overline{\Delta})^{[p]}\neq0$ if and only if $1\leq p\leq\nu(\overline{\Delta})$.

\subsection{Some notions from combinatorial topology}\label{SubSec3}
We recall the notions from combinatorial topology used in Sections~\ref{sec: linearity simpl complex} and~\ref{Sec: characteristic}; for a more detailed account we refer the reader to \cite{Bjorner1995, kozlov2007combinatorial}.

The \textit{geometric realization} of a simplicial complex $\Gamma$ on $[n]$ is the topological space $|\Gamma|=\bigcup_{F\in\Gamma}\operatorname{conv}\{e_i:i\in F\}\subseteq\mathbb{R}^n$, where $e_1,\dots,e_n$ is the standard basis of $\mathbb{R}^n$. With a slight abuse of notation, we use the same symbol $\Gamma$ for the abstract simplicial complex and for its geometric realization, and we attribute to $\Gamma$ the topological properties of $|\Gamma|$. Two topological spaces $X$ and $Y$ are \textit{homotopy equivalent}, written $X\simeq Y$, if there are continuous maps $f\colon X\to Y$ and $g\colon Y\to X$ such that $g\circ f$ is homotopic to the identity of $X$ and $f\circ g$ is homotopic to the identity of $Y$; the space $X$ is \textit{contractible} if it is homotopy equivalent to a point, equivalently, if its identity map is homotopic to a constant map. Since the simplicial homology of $\Gamma$ is isomorphic to the singular homology of $|\Gamma|$, which is invariant under homotopy equivalence, homotopy equivalent simplicial complexes have isomorphic reduced homology groups with coefficients in every commutative ring $R$, and $\widetilde{H}_i(\Gamma;R)=0$ for every $i$ if $\Gamma$ is contractible; see \cite[Chapter~2]{hatcher}. We denote by $S^d$ the $d$-dimensional sphere. A simplicial complex $\Gamma$ is a \textit{triangulation} of a topological space $X$ if $|\Gamma|$ is homeomorphic to $X$, and a \textit{triangulated $d$-sphere} is a triangulation of $S^d$. For instance, the boundary $\partial\Delta^{d+1}$ of the $(d+1)$-simplex, whose faces are the proper subsets of a set with $d+2$ elements, is a triangulated $d$-sphere.

The \textit{join} of two simplicial complexes $\Gamma_1$ and $\Gamma_2$ on disjoint vertex sets is the simplicial complex $\Gamma_1*\Gamma_2=\{F_1\cup F_2:F_1\in\Gamma_1,\ F_2\in\Gamma_2\}$. For a vertex $v$ of no face of $\Gamma$, the \textit{cone} over $\Gamma$ with apex $v$ is $C_v(\Gamma)=\Gamma*\langle\{v\}\rangle$, and for distinct vertices $v,w$ of no face of $\Gamma$, the \textit{suspension} of $\Gamma$ is $\Sigma(\Gamma)=\Gamma*\langle\{v\},\{w\}\rangle$. Every cone is contractible, since the straight-line homotopy from the identity of $|C_v(\Gamma)|$ to the constant map with value $e_v$ is well defined, and the suspension isomorphism gives $\widetilde{H}_{i+1}(\Sigma(\Gamma);R)\cong\widetilde{H}_i(\Gamma;R)$ for every $i$ and every commutative ring $R$ \cite[Chapter~2]{hatcher}.

The \textit{closed star} of a face $G$ of $\Gamma$ is the subcomplex $\st_\Gamma(G)=\{F\in\Gamma:F\cup G\in\Gamma\}$. The \textit{order complex} of a finite poset $P$ is the simplicial complex on $P$ whose faces are the chains of $P$. If some $x\in P$ is comparable with every element of $P$, then the order complex of $P$ is the cone with apex $x$ over the order complex of $P\setminus\{x\}$, and hence it is contractible. The \textit{barycentric subdivision} $\sd(\Gamma)$ of $\Gamma$ is the order complex of the set of nonempty faces of $\Gamma$ ordered by inclusion; its geometric realization is homeomorphic to $|\Gamma|$, so that $\sd(\Gamma)\simeq\Gamma$.

The \textit{nerve} of a finite family $(\Gamma_i)_{i\in I}$ of subcomplexes of $\Gamma$ is the simplicial complex on $I$ whose faces are the sets $J\subseteq I$ such that the subcomplexes $\Gamma_i$, $i\in J$, have a common vertex. We use the following form of the nerve theorem.

\begin{Theorem}[{\cite[Theorem~10.6]{Bjorner1995}}]\label{Thm: nerve}
Let $(\Gamma_i)_{i\in I}$ be a finite family of subcomplexes of a simplicial complex $\Gamma$ such that $\Gamma=\bigcup_{i\in I}\Gamma_i$. Assume that $\bigcap_{i\in J}\Gamma_i$ is contractible for every nonempty set $J\subseteq I$ such that the subcomplexes $\Gamma_i$, $i\in J$, have a common vertex. Then $\Gamma$ is homotopy equivalent to the nerve of $(\Gamma_i)_{i\in I}$.
\end{Theorem}

\section{The Stanley--Reisner complex of squarefree powers}\label{sec: 1-dim flag simplicial complexes}

In this section, we focus on the squarefree powers of the Stanley--Reisner ideal of a $1$-dimensional flag simplicial complex $\Delta$.

Throughout the paper, we denote by $\Delta^{[p]}$ the Stanley--Reisner complex of the squarefree power $I_{\Delta}^{[p]}$, that is, $I_{\Delta^{[p]}}=I_\Delta^{[p]}$. Thus a subset $W\subseteq[n]$ is a face of $\Delta^{[p]}$ if and only if $\overline{\Delta}_W$ has no $p$-matching; in particular, $(\Delta^{[p]})_W=(\Delta_W)^{[p]}$ for every $W\subseteq[n]$. The first question to be addressed is the determination of the Krull dimension of $S/I_\Delta^{[p]}$. We start with an example illustrating the general behaviour.

\begin{Example}

\noindent (1) Let $\Delta_1=\langle \{1,2\},\{1,3\},\{1,4\},\{4,5\}\rangle$ be the $1$-dimensional flag simplicial complex shown in Figure~\ref{fig:exa simpl complex} (A).

\begin{figure}[h]
    \centering
    \subfloat[]{\includegraphics[scale=0.5]{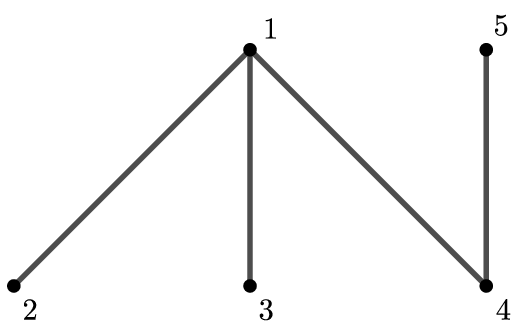}}\qquad\qquad
    \subfloat[]{\includegraphics[scale=0.5]{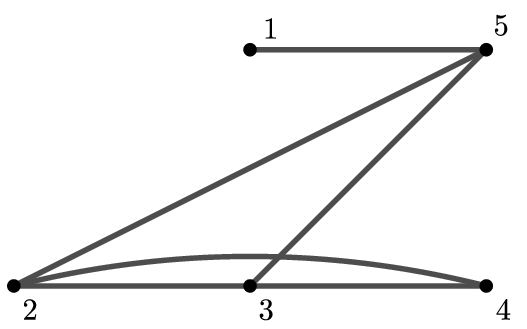}}
    \caption{A $1$-dimensional simplicial complex and its complement graph.}
    \label{fig:exa simpl complex}
\end{figure}

The Stanley--Reisner ideal of $\Delta_1$ is $I_{\Delta_1} = (x_1x_5,\, x_2x_3,\, x_3x_4,\, x_2x_5,\, x_3x_5),$ which can be identified with the edge ideal of the complement graph $\overline{\Delta_1}$ depicted in Figure~\ref{fig:exa simpl complex} (B). The second squarefree power of $I_{\Delta_1}$ is
\[
I_{\Delta_1}^{[2]} = (x_{1}x_{2}x_{3}x_{5},\, x_{1}x_{2}x_{4}x_{5},\, x_{1}x_{3}x_{4}x_{5},\, x_{2}x_{3}x_{4}x_{5}).
\]

The facets of the simplicial complex $\Delta_1^{[2]}$ are
\[
\{
\{3,4,5\}, \{2,4,5\}, \{1,4,5\}, \{2,3,5\},
\{1,3,5\}, \{1,2,5\}, \{1,2,3,4\}
\}.
\]
Note that $\dim S/I_{\Delta_1}^{[2]} = 4 = 2p.$ Moreover, the unique $3$-dimensional facet is $F=\{1,2,3,4\}$, and the induced subcomplex $(\Delta_1)_F$ is isomorphic to the complete bipartite graph $K_{1,3}$.

\medskip

\noindent (2) Let $\Delta_2=\langle\{1,2\},\{3,4\},\{1,4\},\{5\}\rangle$. Then
$$
I_{\Delta_2}^{[2]}
=
\left(
x_{1}x_{2}x_{3}x_{4},\,
x_{1}x_{2}x_{3}x_{5},\,
x_{1}x_{2}x_{4}x_{5},\,
x_{1}x_{3}x_{4}x_{5},\,
x_{2}x_{3}x_{4}x_{5}
\right),
$$
and the facets of $\Delta_2^{[2]}$ are
$$
\left\{
\{3,4,5\},
\{2,4,5\},
\{1,4,5\},
\{2,3,5\},
\{1,3,5\},
\{1,2,5\},
\{2,3,4\},
\{1,3,4\},
\{1,2,4\},
\{1,2,3\}
\right\}.
$$
Note that $\dim S/I_{\Delta_2}^{[2]} = 3 = 2p-1.$ Moreover, in this case, $I_{\Delta_2}^{[2]}$ coincides with the ideal generated by all squarefree monomials of degree $4$ in $K[x_1,\ldots,x_5]$, known as squarefree Veronese ideal generated in degree $4$ in $5$ variables.
\end{Example}

The phenomena observed in the previous examples are not accidental: we show that the dimension can attain only the two values $2p-1$ and $2p$, that it equals $2p$ whenever $\Delta$ contains suitable induced complete bipartite graphs, and that it equals $2p-1$ precisely when $I_{\Delta}^{[p]}$ is a squarefree Veronese ideal.

\begin{Proposition}\label{Thm: dimension}
Let $\Delta$ be a $1$-dimensional flag simplicial complex on $[n]$ and $2\leq p\leq \nu(\overline{\Delta})$. Then:
\begin{enumerate}
\item $\dim S/I_{\Delta}^{[p]}\in \{2p-1,2p\}$.
\item If $\Delta$ contains an induced subgraph isomorphic to $K_{i,\,2p-i}$ for some odd integer $i$ with $1 \le i \le p$, then $\dim S/I_{\Delta}^{[p]} = 2p$.
\end{enumerate}

\end{Proposition}

\begin{proof}

(1) We begin by showing that $\dim S/I_{\Delta}^{[p]} \leq 2p$. If $n=2p$, there is nothing to prove, since $\dim S/I_{\Delta}^{[p]} \leq n = 2p$. We may therefore assume that $n>2p$. Since $\dim S/I_{\Delta}^{[p]}=\max\{s: x_{i_1}\cdots x_{i_s}\notin I_{\Delta}^{[p]},\ i_1<\dots<i_s\}$, it suffices to prove that $x_{i_1}\dots x_{i_{2p}}x_{i_{2p+1}}\in I_{\Delta}^{[p]}$ for every $i_1<\dots<i_{2p}<i_{2p+1}$.

Let $I_1=\{i_1,\dots,i_{2p+1}\}$. We show that the induced graph of $\overline{\Delta}$ on $I_1$ admits a $p$-matching. Suppose first that $\{i_{1},i_k\}\in \Delta$ for every $k=2,\dots,2p+1$. This implies that $\{i_h,i_k\}\notin \Delta$ for every $h,k\in \{2,\dots,2p+1\}$ with $h\neq k$, since otherwise $\Delta$ would contain a triangle. Then
$$\{\{i_j,i_{j+1}\}: j\in\{2,\dots,2p\}\ \text{even}\}$$
forms a $p$-matching of $\overline{\Delta}$, as desired. Otherwise, there exists a vertex in $\{i_2,\dots,i_{2p+1}\}$ that is not adjacent to $i_1$ in $\Delta$. Without loss of generality, we may assume that $\{i_{1},i_{2}\}\notin \Delta$, and we set $I_2=I_1\setminus\{i_1,i_2\}$. If $\{i_{3},i_k\}\in \Delta$ for every $k=4,\dots,2p+1$, then $\{i_h,i_k\}\notin \Delta$ for every $h,k\in \{4,\dots,2p+1\}$ with $h\neq k$, since otherwise $\Delta$ would contain a triangle. Therefore,
$$M_1=\{\{i_j,i_{j+1}\}: j\in\{4,\dots,2p\}\ \text{even}\}$$
forms a $(p-1)$-matching of $\overline{\Delta}$ and, as a consequence, $M_1\cup \{\{i_1,i_2\}\}$ is a $p$-matching of $\overline{\Delta}$. Otherwise, if there exists a vertex, say $i_4$, such that $\{i_{3},i_{4}\}\notin \Delta$, we proceed as above. Iterating this process, we obtain the desired conclusion, namely that the induced graph of $\overline{\Delta}$ on $I_1$ admits a $p$-matching. Hence $x_{i_1}\cdots x_{i_{2p}}x_{i_{2p+1}}\in I_{\Delta}^{[p]}$, which implies $\dim S/I_{\Delta}^{[p]}\leq 2p$, as claimed.

It is straightforward to observe that, if we have a set of $2p-1$ vertices of $\Delta$, namely $i_1,\dots,i_{2p-1}$, then we cannot form a $p$-matching in $\overline{\Delta}$, so $x_{i_1}\cdots x_{i_{2p-1}}\notin I_{\Delta}^{[p]}$, which implies $\dim S/I_{\Delta}^{[p]}\geq 2p-1$. In conclusion, we have $\dim S/I_{\Delta}^{[p]}\in \{2p-1,2p\}$.

(2) Suppose that $\Delta$ contains an induced subgraph $G$ isomorphic to $K_{i, 2p-i}$ for some odd integer $i$ with $1 \le i \le p$. Let $V(G) = V \cup W$, where $V = \{v_1, \dots, v_i\}$ and $W = \{w_1, \dots, w_{2p-i}\}$, and the edge set is defined as $E(G) = \{ \{v_h, w_k\} : h \in [i], k \in [2p-i] \}$. Then the monomial $x_{v_1} \cdots x_{v_i} x_{w_1} \cdots x_{w_{2p-i}}$ does not belong to $I_{\Delta}^{[p]}$ because, by construction, the matching number of the induced subgraph of $\overline{\Delta}$ on $V \cup W$, which is the disjoint union of two complete graphs with $i$ and $2p-i$ vertices, is $(i-1)/2+(2p-i-1)/2=p-1$. Therefore, we conclude that $\dim S/I_{\Delta}^{[p]} = 2p$, since by (1) we have $\dim S/I_{\Delta}^{[p]} \le 2p$.

\end{proof}

As a consequence of the above result, the Castelnuovo--Mumford regularity of $S/I_{\Delta}^{[p]}$ can assume only two possible values, namely $2p-1$ and $2p$.

\begin{Corollary}\label{Coro:reg}
Let $\Delta$ be a $1$-dimensional flag simplicial complex on $[n]$ and $2\leq p\leq \nu(\overline{\Delta})$. Then $\reg S/I_{\Delta}^{[p]} \in \{2p-1,2p\}$.
\end{Corollary}

\begin{proof}
By Proposition~\ref{Thm: dimension}, we have $\reg K[\Delta^{[p]}] \le 2p$. Moreover, since $I_{\Delta}^{[p]}$ is generated in degree $2p$, it follows that $2p-1 \le \reg K[\Delta^{[p]}]$. Therefore, $\reg S/I_{\Delta}^{[p]} \in \{2p-1,2p\}$.
\end{proof}
In the following, we show that $I_{\Delta}^{[p]}$ is a squarefree Veronese ideal when $\dim S/I_{\Delta}^{[p]}$ attains the lower value $2p-1$. Recall that the squarefree Veronese ideal of degree $d$ in $n$ variables, denoted by $I_{n,d}$, is generated by all squarefree monomials of degree $d$ in $S=K[x_1,\dots,x_n]$. As shown in \cite{HHmonomialideals}, the ideal $I_{n,d}$ is the Stanley--Reisner ideal of the $(d-2)$-skeleton of the $(n-1)$-simplex. Consequently, the quotient ring $S/I_{n,d}$ is Cohen--Macaulay with $\dim(S/I_{n,d})=\depth(S/I_{n,d})=d-1$. Moreover, $I_{n,d}$ has linear quotients and, hence, a linear resolution.

\begin{Proposition}\label{prop: veronese}
Let $\Delta$ be a $1$-dimensional flag simplicial complex on $[n]$ and let $2\leq p \leq \nu(\overline{\Delta})$. Then the following conditions are equivalent:
\begin{enumerate}
    \item $\dim S/I_{\Delta}^{[p]}=2p-1$;
    \item $I_{\Delta}^{[p]}$ is a squarefree Veronese ideal generated in degree $2p$ in $n$ variables.
\end{enumerate}
\end{Proposition}

\begin{proof}
$(1)\Rightarrow(2)$. By definition, $I_{\Delta}^{[p]}$ is generated by squarefree monomials of degree $2p$. Moreover,
\[
\dim S/I_{\Delta}^{[p]}
=\max\{s: x_{i_1}\cdots x_{i_s}\notin I_{\Delta}^{[p]},\ i_1<\dots<i_s\}.
\]
Assume by contradiction that the set of generators of $I_{\Delta}^{[p]}$ does not contain all squarefree monomials of degree $2p$. Then there exists a squarefree monomial of degree $2p$ not belonging to $I_{\Delta}^{[p]}$, which implies $\dim S/I_{\Delta}^{[p]}=2p$, contradicting the assumption $\dim S/I_{\Delta}^{[p]}=2p-1$.

\smallskip
$(2)\Rightarrow(1)$. This implication is immediate.
\end{proof}

We now prove the main result of this section. In the case where the Krull dimension of $S/I_{\Delta}^{[p]}$ is $2p$, we describe the $(2p-1)$-dimensional facets of $\Delta^{[p]}$ and characterize when $\Delta^{[p]}$ is pure and when it is Cohen--Macaulay.

For the notation, we set $\mathcal{F}(\Delta, K_{i,j})=\{ F \subseteq [n] \mid \Delta_F \cong K_{i,j}\}$. Moreover, we denote by $\mathcal{F}_i(\Delta)$ the set of all $i$-dimensional facets of $\Delta$.

\begin{Theorem}\label{Thm: Facet, pure, CM}
    Let $\Delta$ be a 1-dimensional flag simplicial complex on $[n]$ and $2\leq p\leq \nu(\overline{\Delta})$. Assume that $\dim S/I_{\Delta}^{[p]}=2p$. Then
    \begin{enumerate}
   \item The $(2p-1)$-dimensional facets of $\Delta^{[p]}$ are given by
   \[
   \mathcal{F}_{2p-1}(\Delta^{[p]})=
   \begin{cases}
       \mathcal{F}( \Delta, K_{1,2p-1}) \cup \mathcal{F}( \Delta, K_{3,2p-3}) \cup \ldots \cup \mathcal{F}( \Delta, K_{p,p}) & \text{if } p \text{ is odd } \\
     \mathcal{F}( \Delta, K_{1,2p-1}) \cup \mathcal{F}( \Delta, K_{3,2p-3}) \cup \ldots \cup \mathcal{F}( \Delta, K_{p-1,p+1})   & \text{if } p \text{ is even}
   \end{cases}
\]

   \item $\Delta^{[p]}$ is pure if and only if one of the following holds:
     \begin{enumerate}
       \item[(i)] $\Delta \iso K_{2k+1,2\ell+1}$ and $n=(2k+1)+(2\ell+1)\ge 2p$,
       \item[(ii)] $\Delta \iso K_{2k+1,2\ell}$ and $\ell \geq p$,
        \item[(iii)] $\Delta \iso K_{2k,2\ell}$ and $k,\ell \geq p$.
    \end{enumerate}

\item $\Delta^{[p]}$ is Cohen--Macaulay if and only if $\Delta \iso K_{1,n-1}$ and $n>2p$.
    \end{enumerate}
\end{Theorem}

\begin{proof}
We divide the proof into three parts, corresponding to the statements above.

\textbf{(1)} First, we show by induction on $p$ that $\Delta_F$ is bipartite for any facet $F$ of $\Delta^{[p]}$ of dimension $2p-1$. Let $F$ be a $(2p-1)$-dimensional facet of $\Delta^{[p]}$. In the case $p=2$, $\Delta_F$ is trivially bipartite since $\Delta_F$ does not contain any triangle. Now assume $p \ge 3$. Suppose by contradiction that $\Delta_F$ is not bipartite. Then $\Delta_F$ contains an odd cycle $C_{2s+1}$ as an induced subgraph, for some $s\geq 1$.

\textit{Case 1.} Suppose that $2s+1 = 2p-1$. Set $F \setminus V(C_{2s+1}) = \{v\}$. Since $\Delta$ does not contain triangles, there exists $w \in V(C_{2s+1})$ such that $\{v,w\} \notin \Delta$. Since the vertices of $C_{2s+1}$ different from $w$ induce in $\Delta$ a path with $2p-2\geq4$ vertices, whose complement has a perfect matching, we have
$$
\prod_{u \in V(C_{2s+1}) \setminus \{w\}} x_u\in I_{\Delta}^{[p-1]}.
$$
Hence,
$$
x_vx_w \prod_{u \in V(C_{2s+1}) \setminus \{w\}} x_u\in I_{\Delta}^{[p]}.
$$
Therefore, $F$ cannot be a facet of $\Delta^{[p]}$, which is a contradiction.

\textit{Case 2.} Suppose that $2s+1 \le 2p-3$. Then there exist $v_1, v_2, v_3$ in $F \setminus V(C_{2s+1})$. Since $\Delta$ does not contain triangles, we may assume that $\{v_1,v_2\} \notin \Delta$. Since $\Delta_{F\setminus\{v_1,v_2\}}$ contains $C_{2s+1}$, it is not bipartite, and the inductive hypothesis gives
$$
\prod_{u \in F \setminus \{v_1, v_2\}} x_u \in I_{\Delta}^{[p-1]}.
$$
Hence
$$
x_{v_1} x_{v_2} \prod_{u \in F \setminus \{v_1, v_2\}} x_u \in I_{\Delta}^{[p]}.
$$
Therefore, $F$ cannot be a facet of $\Delta^{[p]}$, which is a contradiction.

Now we know that $\Delta_F$ is bipartite and is contained, as a spanning subgraph, in a complete bipartite graph $K_{\ell,m}$, where $\ell$ and $m$ are either both even or both odd. If both $\ell$ and $m$ are even, the complement graph $\overline{K_{\ell,m}}$ has a $p$-matching. Hence so does $\overline{\Delta_F}$. Therefore, $F$ cannot be a facet of $\Delta^{[p]}$. If both $\ell$ and $m$ are odd and $\Delta_F$ is not the complete bipartite graph $K_{\ell,m}$, then the complement graph $\overline{\Delta_F}$ has a $p$-matching: match a pair $\{x,y\}\notin\Delta$ with $x$ and $y$ in different parts, and pair up the remaining vertices of each part, which are in even number. Hence $F$ cannot be a facet of $\Delta^{[p]}$. Thus, if $F$ is a facet of $\Delta^{[p]}$, then $\Delta_F$ must be the complete bipartite graph $K_{\ell,m}$ with $\ell$ and $m$ odd.

On the other hand, $\overline{K_{\ell,m}}$ with $\ell$ and $m$ odd does not have a $p$-matching. Hence, if $\Delta_F$ is $K_{\ell,m}$ with $\ell$ and $m$ odd, then $F$ is a facet of $\Delta^{[p]}$.

\textbf{(2)} We assume that $\Delta^{[p]}$ is pure. We show that $\Delta$ is bipartite. Suppose that $\Delta$ contains an odd cycle $C_{2s+1}$ with $s \ge 2$. Then we can take $W \subset [n]$ with $|W| = 3$ such that $\Delta_W$ is the disjoint union of an edge and an isolated vertex. However, this cannot be contained as an induced subgraph in a complete bipartite graph $K_{\ell,m}$ with $\ell + m = 2p$ and $\ell, m$ odd. Hence we get a contradiction and, then, $\Delta$ has to be bipartite.

Next, we show that $\Delta$ is complete bipartite. Suppose that $\Delta$ is contained in a complete bipartite graph $K_{s,t}$ with $X \sqcup Y$ as vertex set partition, and that $\Delta$ is not complete bipartite. We may assume that $|X| \ge 2$. Let $\{x, y\} \notin \Delta$, $\{x', y\} \in \Delta$, where $x, x' \in X$, $y \in Y$. Take subsets $X' \subset X$, $Y' \subset Y$ such that $x, x' \in X'$, $y \in Y'$, and $|X'| + |Y'| = 2p-1$. Then $\Delta_{X' \cup Y'}$ is neither a complete bipartite graph nor an edgeless graph. Hence $\Delta_{X' \cup Y'}$ cannot be an induced subgraph of a complete bipartite graph. Thus $X' \cup Y'$ is not contained in a facet of dimension $2p-1$. Therefore, $\Delta$ is complete bipartite.

Since $p \leq \nu(\overline{\Delta})$, we have $n \ge 2p$.

If $\Delta \cong K_{2k+1,2\ell}$ with $\ell < p$ and $k > 0$, then $K_{2p-2\ell-1,2\ell}$ cannot be an induced subgraph of the form $K_{2s+1, 2t+1}$ with $(2s+1) + (2t+1) = 2p$. Hence $V(K_{2p-2\ell-1,2\ell})$ is a facet of $\Delta^{[p]}$ of dimension $2p-2$. Therefore $\Delta^{[p]}$ is not pure of dimension $2p-1$. Hence $\ell \ge p$ is necessary.

If $\Delta \cong K_{1,2\ell}$ with $\ell < p$, then $2\ell + 1 \le 2p-1$, which contradicts the assumption that $p \leq \nu(\overline{\Delta})$. Hence $\ell \ge p$ is necessary.

If $\Delta \cong K_{2k,2\ell}$ with $k \le \ell$ and $k < p$, then $K_{2k,2p-2k-1}$ cannot be an induced subgraph of the form $K_{2s+1, 2t+1}$ with $(2s+1) + (2t+1) = 2p$. Hence $V(K_{2k,2\ell})$ is a facet of $\Delta^{[p]}$ of dimension $2p-2$. Therefore $\Delta^{[p]}$ is not pure of dimension $2p-1$. Hence $k, \ell \ge p$ is necessary.

Next, we show the converse. For any facet $F$ of $\Delta^{[p]}$ with $|F| = 2p-1$, $\Delta_F$ is a complete bipartite graph of the form $K_{2s,2t+1}$ with $2s + (2t+1) = 2p-1$, which is an induced subgraph of $K_{2s+1,2t+1}$, whose vertex set is a facet of $\Delta^{[p]}$. Thus $\Delta^{[p]}$ is pure.

\textbf{(3)} Assume that $\Delta^{[p]}$ is Cohen--Macaulay. Then $\Delta^{[p]}$ is pure. By (2), $\Delta$ is isomorphic to $K_{s,t}$ for some $s,t > 0$ with $s + t = n$.

Assume that $s,t > 1$ and that both $s$ and $t$ are odd. Then $n=s+t$ is even, and $n\neq 2p$, since otherwise $\nu(\overline{\Delta})=(s-1)/2+(t-1)/2=p-1$; hence $n\geq 2p+2$, and we can take odd integers $3 \le \ell \le s$ and $3 \le m \le t$ such that $\ell + m = 2p+2$. Then there exist $F \in \Delta^{[p]}$ such that $\Delta_F \cong K_{\ell-2,m}$ and $G \in \Delta^{[p]}$ such that $\Delta_G \cong K_{\ell,m-2}$. Then $F$ and $G$ cannot be connected in codimension one, since two $(2p-1)$-dimensional faces sharing $2p-1$ vertices have the same number of vertices in each part.

Assume that $s \ge 2p$ is even and $t > 1$, so that $t\geq3$ by (2). Then there exist $F \in \Delta^{[p]}$ such that $\Delta_F \cong K_{2p-1,1}$ and $G \in \Delta^{[p]}$ such that $\Delta_G \cong K_{2p-3,3}$. Then $F$ and $G$ cannot be connected in codimension one.

Hence, if $\Delta^{[p]}$ is Cohen--Macaulay, then $\Delta \cong K_{1,n-1}$; moreover, if $n = 2p$, then $\nu(\overline{\Delta}) = p-1$, a contradiction, so $n > 2p$. Conversely, let $\Delta\cong K_{1,n-1}$ with centre $c$ and $n>2p$. A subset $W\subseteq[n]$ is a face of $\Delta^{[p]}$ if and only if $|W\setminus\{c\}|\leq2p-1$, since $\overline{\Delta}_W$ is the disjoint union of a complete graph on $W\setminus\{c\}$ and, possibly, the isolated vertex $c$. Hence $\Delta^{[p]}$ is the cone with apex $c$ over the $(2p-2)$-skeleton of the simplex on $[n]\setminus\{c\}$, which is Cohen--Macaulay, and so is $\Delta^{[p]}$.
\end{proof}

\begin{Corollary}\label{coro: facet}
Let $\Delta$ be a $1$-dimensional flag simplicial complex on $[n]$ and $2\leq p\leq \nu(\overline{\Delta})$. Then, $\dim S/I_{\Delta}^{[p]}= 2p$ if and only if $\Delta$ contains an induced subgraph isomorphic to $K_{i,\,2p-i}$ for some odd integer $i$ with $1 \le i \le p$.
\end{Corollary}

\begin{proof}
    It follows from Proposition  \ref{Thm: dimension} and Theorem \ref{Thm: Facet, pure, CM}.
\end{proof}
We conclude the section by giving the values of the projective dimension.
\begin{Theorem}\label{Thm: pd}
Let $\Delta$ be a $1$-dimensional flag simplicial complex on $[n]$, and let $2 \leq p \leq \nu(\overline{\Delta})$. If $\overline{\Delta}$ has a unique $p$-matching, then $I_{\Delta}^{[p]}$ is principal and $\operatorname{pd}(S/I_{\Delta}^{[p]})=1$. If $\overline{\Delta}$ has more than one $p$-matching, then
\[
\operatorname{pd}(S/I_{\Delta}^{[p]})=
\begin{cases}
n-2p, & \text{if } \Delta \cong K_{1,n-1},\\
n-2p+1, & \text{otherwise}.
\end{cases}
\]
\end{Theorem}

\begin{proof}
    Assume that $\Delta \cong K_{1,n-1}$. Since $2 \leq p \leq \nu(\overline{\Delta})$, we necessarily have $n\geq 2p+1$. By Theorem~\ref{Thm: Facet, pure, CM}(3), the ring $S/I_{\Delta}^{[p]}$ is Cohen--Macaulay. Moreover, $\Delta$ contains an induced subgraph isomorphic to $K_{1,\,2p-1}$, so $\dim S/I_{\Delta}^{[p]}= 2p$. Therefore from the Auslander--Buchsbaum formula, we get $\operatorname{pd}(S/I_{\Delta}^{[p]}) = n-2p$.

If $\Delta$ does not contain an induced subgraph isomorphic to $K_{i,\,2p-i}$ for any odd integer $i$ with $1 \le i \le p$, then by Corollary \ref{coro: facet} and Proposition \ref{Thm: dimension} (1), we have $\operatorname{dim}(S/I_{\Delta}^{[p]}) = 2p-1$. Hence, from Proposition \ref{prop: veronese}, $I_{\Delta}^{[p]}$ is the squarefree Veronese ideal generated in degree $2p$ in $n$ variables, so $S/I_{\Delta}^{[p]}$ is Cohen--Macaulay of dimension $2p-1$, as recalled before Proposition~\ref{prop: veronese}, and the Auslander--Buchsbaum formula gives $\operatorname{pd}(S/I_{\Delta}^{[p]}) = n-2p+1$.

Now assume that $\Delta$ is not isomorphic to $K_{1,n-1}$ and $\Delta$ contains an induced subgraph isomorphic to $K_{i,\,2p-i}$ for some odd integer $i$ with $1 \le i \le p$. Hence  $\operatorname{dim}(S/I_{\Delta}^{[p]}) = 2p$ and $S/I_{\Delta}^{[p]}$ is not Cohen--Macaulay by Theorem~\ref{Thm: Facet, pure, CM}(3), that is, $\operatorname{dim}(S/I_{\Delta}^{[p]})>\operatorname{depth}(S/I_{\Delta}^{[p]})$. Therefore, from the Auslander--Buchsbaum formula, it follows that $\operatorname{pd}(S/I_{\Delta}^{[p]}) > n-2p$. On the other hand, since $I_{\Delta}^{[p]}$ is a squarefree monomial ideal generated in degree $2p$ in $n$ variables, Hochster's formula yields $\operatorname{pd}(I_{\Delta}^{[p]}) \le n-2p$, because $\widetilde{H}_k(\Delta^{[p]}_W;K)=0$ for $k<2p-2$, that is, $\operatorname{pd}(S/I_{\Delta}^{[p]}) \le n-2p+1$. Hence  $\operatorname{pd}(S/I_{\Delta}^{[p]}) = n-2p+1$.
\end{proof}

\section{Linearly related squarefree powers of edge ideals}\label{Sec: linearly related}

The main goal of this section is to characterize the graphs $G$ for which the squarefree power $I(G)^{[p]}$ is linearly related. We show that this property admits a simple graph-theoretic description in terms of induced subgraphs.

\begin{Theorem}\label{Thm:linrelated1}
Let $G$ be a graph and $2 \leq p \leq \nu(G)$. Then $I(G)^{[p]}$ is linearly related if and only if there does not exist an induced subgraph $H$ of $G$ on $2p+2$ vertices such that $H$ is disconnected and $\nu(H) = p+1$.
\end{Theorem}

Before turning to the proof of Theorem~\ref{Thm:linrelated1}, we present an example illustrating the theorem and record an immediate consequence for Stanley--Reisner ideals of $1$-dimensional flag simplicial complexes.

\begin{Example}\label{exa1}
Consider the graph $G_1$ in Figure \ref{fig: graphs} (A). Its edge ideal is
\[
I(G_1)=(x_1x_2,x_1x_5,x_2x_5,x_3x_5,x_3x_4,x_4x_5,x_5x_6,x_5x_7,x_6x_7),
\]
and its second squarefree power is generated by
\[
\begin{aligned}
I(G_1)^{[2]} = (&x_1x_2x_3x_4,\, x_1x_2x_3x_5,\, x_1x_2x_4x_5,\, x_1x_3x_4x_5,\, x_2x_3x_4x_5,\, x_1x_2x_5x_6,\, x_3x_4x_5x_6,\\
&x_1x_2x_5x_7,\,x_3x_4x_5x_7,\,x_1x_2x_6x_7,\,x_3x_4x_6x_7,\,x_1x_5x_6x_7,\,x_2x_5x_6x_7,\,x_3x_5x_6x_7,\,x_4x_5x_6x_7).
\end{aligned}
\]

According to Theorem \ref{Thm:linrelated1}, this ideal is not linearly related, since the induced subgraph of $G_1$ on the vertex set $\{1,2,3,4,6,7\}$ is disconnected and has matching number $3$. By using \textit{Macaulay2} \cite{M2}, its Betti table over a field of characteristic zero is the following:
\[
\begin{array}{r|cccc}
      & 0 & 1 & 2 & 3 \\
\hline
\mathrm{total:}  & 15 & 26 & 15 & 3\\
0:  & . & . & . & .\\
1:  & . & . & . & .\\
2:  & . & . & . & .\\
3:  & . & . & . & .\\
4:  & 15 & 24 & 12 & 3\\
5:  & . & 2 & 3 & .
\end{array}
\]

On the other hand, the second squarefree power of the edge ideal of $G_2$ is linearly related, since $G_2$ contains no induced subgraph on $6$ vertices that is disconnected and has matching number $3$.

\begin{figure}[h]
    \centering
    \subfloat[$G_1$]{\includegraphics[scale=0.6]{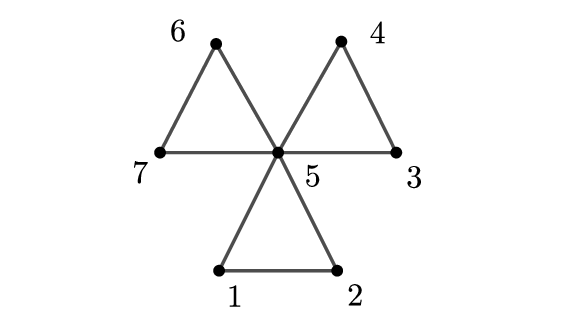}}\qquad
    \subfloat[$G_2$]{\includegraphics[scale=0.6]{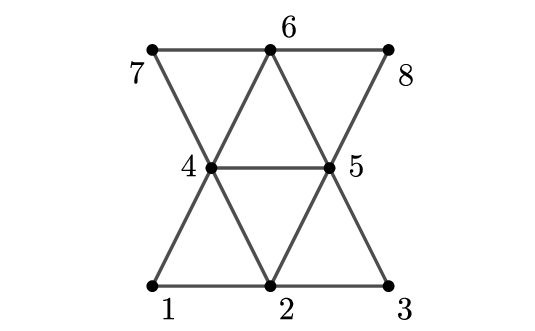}}
    \caption{The graphs of Example~\ref{exa1}.}
    \label{fig: graphs}
\end{figure}

\end{Example}

Since $I_{\Delta}=I(\overline{\Delta})$, Theorem~\ref{Thm:linrelated1} has the following consequence for $1$-dimensional flag simplicial complexes, which we prove after Theorem~\ref{Thm:linrelated1}.

\begin{Corollary}\label{coro: lin related}
Let $\Delta$ be a $1$-dimensional flag simplicial complex on $[n]$, and let $2 \leq p \leq \nu(\overline{\Delta})$. The following conditions are equivalent:
\begin{enumerate}
    \item $I_{\Delta}^{[p]}$ is linearly related.
    \item $\Delta$ does not contain an induced subgraph isomorphic to $K_{2+i,\,2p-i}$ for any even integer $i$ with $0 \leq i \leq p$.
\end{enumerate}

\end{Corollary}


To prove Theorem~\ref{Thm:linrelated1}, we first recall a connectivity criterion from \cite{BHZ} and establish some preliminary results.

 Let $I$ be a monomial ideal generated in degree $d$. We denote by $\cG(I)$ the minimal set of monomial generators of $I$. In \cite{BHZ}, the authors associated a graph $G_I$ to $I$ as follows: $V(G_I)=\cG(I)$, and  $\{u,v\} \in E(G_I)$ if and only if $\deg(\mathrm{lcm}(u,v))=d+1$. Moreover, for all $u,v\in \cG(I)$, the induced subgraph of $G_I$ on the vertex set $\{w\in V(G_I)\mid w\text{ divides } \lcm(u,v)\}$ is denoted by $G^{(u,v)}_{I}$. The connectivity of $G^{(u,v)}_{I}$ characterizes when $I$ is linearly related, as shown in the following.

		\begin{Theorem}\cite[Corollary 2.2]{BHZ}\label{ref:BHN}
		   Let $I$ be a monomial ideal generated in degree $d$. Then $I$ is linearly related if and only if for all $u,v\in \cG(I)$ there is a path in $G^{(u,v)}_{I}$ connecting $u$ and $v$.
		\end{Theorem}


Before proving the next proposition, we recall some basic
terminology and facts from matching theory; see
\cite[Chapter 1]{LP86}. Let $M$ and $N$ be two matchings of $G$. Their
\emph{symmetric difference} is defined by
\[
M\triangle N=(M\setminus N)\cup(N\setminus M).
\]
A path (respectively, cycle) is called \emph{alternating}
with respect to $M$ and $N$ if its edges alternate between
$M\setminus N$ and $N\setminus M$.

Since each vertex is incident to at most one edge of each
matching, every vertex of the graph with edge set
$M\triangle N$ has degree at most two. Consequently,
its nontrivial connected components are alternating paths
or even alternating cycles.


\begin{Remark}\label{rm:perfect} A matching $M$ of a graph $G$ is called
\emph{perfect} if every vertex of $G$ belongs to an edge of $M$. Let $G$ be a graph on $2p+2$ vertices, let $M$ be a perfect matching of $G$ and $N$ be a matching of size $p$.
Then $N$ leaves exactly two vertices uncovered.
These are precisely the two vertices of degree one
in $M\triangle N$. Hence $M\triangle N$ consists of
a unique alternating path joining these two vertices,
together with some alternating cycles.
The first and last edges of this path belong to $M$.
\end{Remark}


\begin{Proposition}\label{prop:connected}
    Let $G$ be a connected graph on $2p+2$ vertices. Then $I(G)^{[p]}$ is linearly related.
\end{Proposition}
\begin{proof}
Set $J=I(G)^{[p]}$. If $\nu(G) <p$, then $J=0$ and there is nothing to prove. Since $G$ has $2p+2$ vertices and
$p\leq \nu(G)$, we have $\nu(G)\in\{p,p+1\}$. If $\nu(G)=p$, then $J$ is the highest nonzero squarefree
power of $I(G)$. By \cite[Theorem 5.1]{BHZ}, $J$
has linear quotients and hence is linearly related.

Now suppose that $\nu(G)=p+1$, and fix a perfect matching
$M$ of $G$. Write $V=V(G)$ and $x_V=\prod_{v\in V}x_v$.
For every edge $e\in M$, define $u_e=\frac{x_V}{x_e}.$ Clearly, $u_e\in \mathcal G(J)$, since it corresponds to the
$p$-matching $M\setminus\{e\}$. We first show that every generator of $J$ is connected
to some $u_e$ in the associated graph $G_J$.

Let $u\in \mathcal G(J)$, and let $N$ be a $p$-matching corresponding to $u$. Since $M$ is perfect, by Remark~\ref{rm:perfect} $N$ leaves exactly two vertices
uncovered and  the symmetric difference $M\triangle N$
consists of a unique alternating path joining these two
vertices, together with some alternating cycles.
The alternating cycles do not affect the set of vertices
covered by either matching.

Write the alternating path as $P:v_0,v_1,\ldots,v_{2r+1},$ where its edges belonging to $M$ are $\{v_{2i},v_{2i+1}\}$ for $0\leq i\leq r,$ and those belonging to $N$ are $\{v_{2i-1},v_{2i}\}$ for $1\leq i\leq r.$ In particular, $v_0$ and $v_{2r+1}$ are the two vertices
not covered by $N$. If $r=0$, these two vertices form an edge $e\in M$,
and hence $u=u_e$. Suppose $r\geq 1$. Starting with $N_0=N$, define
successively, for $i=1,\ldots,r$,
\[
N_i=
\bigl(N_{i-1}\setminus
\{\{v_{2i-1},v_{2i}\}\}\bigr)
\cup\{\{v_{2i-2},v_{2i-1}\}\}.
\]
At the $i$-th step, the vertex $v_{2i-2}$ is uncovered
by $N_{i-1}$, while $v_{2i-1}$ is matched with
$v_{2i}$. Therefore, $N_i$ is again a $p$-matching. For each $0\leq i\leq r$, let $u_i$ be the generator of $J$ corresponding to the $p$-matching $N_i$, that is, $u_i=\prod_{\{a,b\}\in N_i}x_ax_b.$ Since $N_i$ is obtained from $N_{i-1}$ by replacing
one covered vertex with another, we have $\deg\bigl(\operatorname{lcm}(u_{i-1},u_i)\bigr)
=2p+1.$ Thus $u_{i-1}$ and $u_i$ are adjacent in $G_J$. Consequently, $u=u_0,u_1,\ldots,u_r$ is a path in $G_J$. After the last replacement, the two uncovered vertices are $v_{2r}$ and $v_{2r+1}$. Since
$e=\{v_{2r},v_{2r+1}\}\in M$, we obtain $u_r=u_e.$ Hence every generator of $J$ is connected to some $u_e$ with $e\in M$.

We next show that all such generators $u_e$ belong
to the same connected component of $G_J$.

Contract every edge of the perfect matching $M$ to
a single vertex. Since $G$ is connected, the resulting
graph is also connected. Suppose that two edges $e,f\in M$ correspond to
adjacent vertices in the contracted graph. Write
$e=\{a,b\}$ and $f=\{c,d\}$. Then there is an edge
of $G$ joining a vertex of $e$ to a vertex of $f$.
Without loss of generality, assume that
$\{a,c\}\in E(G)$.

Consider $w=x_V/(x_bx_d).$ The monomial $w$ belongs to $\mathcal G(J)$, since it
corresponds to the $p$-matching $(M\setminus\{e,f\})\cup\{\{a,c\}\}.$ Moreover, $\lcm(u_e,w)=x_V/x_b,$ and $\lcm(w,u_f)=x_V/x_d.$ Both least common multiples have degree $2p+1$. Therefore, $u_e, w, u_f$ is a path in $G_J$. Since the contracted graph is connected, any two
edges of $M$ can be joined by a sequence of such
adjacent edges. It follows that all generators
$u_e$, with $e\in M$, are connected in $G_J$.

Combining the two arguments above, we conclude
that the graph $G_J$ is connected. Finally, let $u,v\in \mathcal G(J)$ be distinct generators and set $m=\operatorname{lcm}(u,v)$. Since both
generators are squarefree of degree $2p$ and
$|V|=2p+2$, we have $\deg(m)\in\{2p+1,2p+2\}.$ If $\deg(m)=2p+1$, then $u$ and $v$ are adjacent in $G_J^{(u,v)}$. If $\deg(m)=2p+2$, then $m=x_V$. Hence every
minimal generator of $J$ divides $m$, and therefore $G_J^{(u,v)}=G_J.$ As we have already proved that $G_J$ is connected, there exists a path connecting $u$ and $v$ in
$G_J^{(u,v)}$. Thus, by Theorem~\ref{ref:BHN}, $J=I(G)^{[p]}$ is linearly
related.
\end{proof}



	\begin{Proposition}\label{thm:betti}
		Let $G$ be a graph and $2\leq p \leq \nu(G)$. Then $\beta_{1,j}(I(G)^{[p]}) = 0$ if $j \notin \{ 2p+1,2p+2\}$.
	\end{Proposition}

\begin{proof}
Let $I=I(G)$. Consider the Taylor complex associated with
$I^{[p]}$ (see \cite[Section~7.1]{HHmonomialideals}).
From its construction, we observe that
$\beta_{1,m}(I^{[p]})=0$ unless
$m=\lcm(u,v)$ for some $u,v\in\mathcal{G}(I^{[p]})$.
Furthermore, by a result of Gasharov, Peeva, and
Welker~\cite{Gasharov1999}, for every $m$ in the
lcm-lattice $L(I^{[p]})$, we have
\[
\beta_{1,m}(I^{[p]})
=\dim_K\widetilde{H}_0(\Delta((1,m));K),
\]
where $\Delta((1,m))$ denotes the order complex of
the open interval $(1,m)$ in $L(I^{[p]})$. Since the least common multiple of two generators
of $I^{[p]}$ has degree at most $4p$, it suffices
to show that $(1,m)$ is connected whenever
$m$ is the least common multiple of two generators
and $2p+2<\deg(m)\leq 4p.$

Let $P=(1,m)$ and denote by $G_P$ its comparability
graph. Recall that two vertices of $G_P$ are adjacent
if they are comparable in $P$.
Since every element of $P$ is divisible by some
minimal generator of $I^{[p]}$ dividing $m$,
it is enough to show that any two distinct generators $u,v$ dividing $m$ are connected in $G_P$. If $\lcm(u,v)$ strictly divides $m$, then $u$ and $v$ are connected through their least common multiple. Otherwise, choose $p$-matchings $M$ and $N$ corresponding to $u$ and $v$, respectively.

Consider the bipartite graph $H$ with vertex set
$(M\setminus N)\cup(N\setminus M)$, where
$e\in M\setminus N$ and $f\in N\setminus M$
are adjacent if $e\cap f\neq\emptyset$.
Since $M$ and $N$ are matchings, every vertex
of $H$ has degree at most two. Moreover, there exists $e\in M\setminus N$
having at most one neighbor in $H$. Indeed,
otherwise every vertex corresponding to an edge
of $M\setminus N$ would have degree two.
Since $|M\setminus N|=|N\setminus M|$,
every vertex of $H$ would then have degree two.
This would imply that $M$ and $N$ cover the
same vertices, contradicting $u\neq v$.

If $e$ has a unique neighbor $f\in N\setminus M$,
replace $f$ by $e$ in $N$. If $e$ has no neighbor,
we may replace any edge $f\in N\setminus M$
by $e$. In either case, $N'=(N\setminus\{f\})\cup\{e\}$ is a $p$-matching, since $e$ is disjoint from
all edges of $N\setminus\{f\}$. Let $v'$ be the generator corresponding to $N'$.
Both $v$ and $v'$ divide $m$, and
\[
\deg(\lcm(v,v'))
\leq 2p+2<\deg(m).
\]
Thus $v$ and $v'$ are connected in $G_P$.

Moreover, $N'$ has one more edge in common with
$M$ than $N$ does. Repeating this procedure
as long as the corresponding generators are
distinct, we eventually obtain a matching
whose covered vertex set coincides with that of
$M$, and hence whose corresponding generator
is $u$. Therefore, $u$ and $v$ are connected
in $G_P$.

Consequently, $P$ is connected, and
$\beta_{1,m}(I^{[p]})=0$ whenever
$\deg(m)>2p+2$. Since $I^{[p]}$ is generated
in degree $2p$, the assertion follows.
\end{proof}


\begin{proof}[Proof of Theorem~\ref{Thm:linrelated1}]
We show that the following statements are equivalent: (1) $I(G)^{[p]}$ is not linearly related; (2) there exists an induced subgraph $H$ of $G$ on $2p+2$ vertices such that $H$ is disconnected and $\nu(H)=p+1$.

$(2)\Rightarrow(1)$.
Suppose that $G$ contains a disconnected induced
subgraph $H$ on $2p+2$ vertices with $\nu(H)=p+1$.
Fix a perfect matching $M$ of $H$.
Since $H$ is disconnected, we may write
$H=H_1\sqcup H_2$, where $H_1$ and $H_2$ are
nonempty unions of connected components of $H$.
As $M$ is perfect, both $H_1$ and $H_2$ have an
even number of vertices, say $2a$ and $2b$,
respectively.

Every generator of $I(H)^{[p]}$
corresponds to a $p$-matching of $H$ and therefore
leaves exactly two vertices of $V(H)$ uncovered.
Since each matching covers an even number of
vertices in each $H_i$, the two uncovered vertices
must belong to the same $H_i$.

Set $V=V(H)$ and $x_V=\prod_{i\in V}x_i$.
Let $A$ and $B$ be the sets of generators of
$I(H)^{[p]}$ whose uncovered vertices belong to
$H_1$ and $H_2$, respectively.
Choose $e\in M\cap E(H_1)$ and
$f\in M\cap E(H_2)$. Then
$u=x_V/x_e\in A$ and $v=x_V/x_f\in B$,
and hence both sets $A$ and $B$ are nonempty.

Moreover, for any $u'\in A$ and $v'\in B$, their
uncovered vertex sets are disjoint, and hence
$\lcm(u',v')=x_V$, with degree $2p+2$.
Thus no generator in $A$ is adjacent to a
generator in $B$, so the associated graph
$G_{I(H)^{[p]}}$ is disconnected.

Also, note that $\lcm(u,v)=x_V$.
Since $H$ is an induced subgraph of $G$, the generators
of $I(G)^{[p]}$ dividing $x_V$ are precisely the
generators of $I(H)^{[p]}$. Consequently,
$G_{I(G)^{[p]}}^{(u,v)}=G_{I(H)^{[p]}}$.
In particular, $u$ and $v$ are not connected
in $G_{I(G)^{[p]}}^{(u,v)}$. Therefore,
$I(G)^{[p]}$ is not linearly related by
Theorem~\ref{ref:BHN}.

$(1)\Rightarrow(2)$.
Let $I=I(G)$ and assume that $I^{[p]}$ is not linearly
related. By Proposition~\ref{thm:betti}, we have
$\beta_{1,2p+2}(I^{[p]})\neq 0$.
By the lcm-lattice formula used in its proof,
there exists a monomial $m$ of degree $2p+2$
such that the open interval $P=(1,m)$ is disconnected.
Hence there exist two generators
$u,v\in\mathcal{G}(I^{[p]})$ belonging to different
connected components of the comparability graph $G_P$.
In particular, $m=\lcm(u,v)$.

Let $H$ be the induced subgraph of $G$ on
$\supp(u)\cup\supp(v)$, and set $J=I(H)$.
Then $H$ has $2p+2$ vertices. Moreover,
$u,v\in\mathcal{G}(J^{[p]})$, and
$\beta_{1,2p+2}(J^{[p]})\neq 0$, because the
interval $(1,m)$ is the same in $L(J^{[p]})$
and $L(I^{[p]})$. Thus $J^{[p]}$ is not linearly
related. By Proposition~\ref{prop:connected},
it follows that $H$ is disconnected.

It remains to show that $\nu(H)=p+1$.
Since $J^{[p]}\neq 0$ and $H$ has $2p+2$ vertices,
we have $p\leq\nu(H)\leq p+1$.
If $\nu(H)=p$, then $J^{[p]}$ is the highest
nonzero squarefree power of $J$, which has
a linear resolution by \cite[Theorem~5.1]{BHZ}.
This contradicts the fact that $J^{[p]}$
is not linearly related. Therefore,
$\nu(H)=p+1$, as required.
\end{proof}

We can now give a proof of Corollary~\ref{coro: lin related}.

\begin{proof}[Proof of Corollary \ref{coro: lin related}]
$(1) \Rightarrow (2)$. If $\Delta_W\cong K_{2+i,\,2p-i}$ for some $W\subseteq[n]$ and some even integer $0 \leq i \leq p$, then $\overline{\Delta}_W$ is the disjoint union of two complete graphs with $2+i$ and $2p-i$ vertices, so it is disconnected and $\nu(\overline{\Delta}_W)=p+1$. By Theorem~\ref{Thm:linrelated1}, $I_{\Delta}^{[p]}$ is not linearly related.

$(2) \Rightarrow (1)$. Assume that $I_{\Delta}^{[p]}$ is not linearly related. By Theorem~\ref{Thm:linrelated1}, there exists an induced subgraph $H$ of $\overline{\Delta}$ on $2p+2$ vertices such that $H$ is disconnected and $\nu(H)=p+1$; let $M$ be a perfect matching of $H$. Since $\Delta$ has no triangles, $\overline{\Delta}$ has no independent set of size three, so $H$ has exactly two connected components $H_1$ and $H_2$. Moreover, if $v_r,v_s\in V(H_1)$ are distinct and $v_t \in V(H_2)$, then $\{v_r,v_t\},\{v_s,v_t\}\notin E(H)$, hence $\{v_r,v_s\} \in E(H_1)$. Thus $H_1$ and $H_2$ are complete graphs, with $2j$ and $2p+2-2j$ vertices respectively, where $j=|M\cap E(H_1)|\geq1$. Up to exchanging $H_1$ and $H_2$ we may assume $2j\leq p+1$, and then $\Delta_{V(H)}\cong K_{2j,\,2p+2-2j}=K_{2+i,\,2p-i}$ with $i=2j-2$ even and $0\leq i\leq p$.
\end{proof}

\begin{Example}\label{exa: K24 count}
Consider the $1$-dimensional flag simplicial complex $\Delta$ in Figure \ref{fig: 1-dimensional flag simplicial complex}.

\begin{figure}[h]
    \centering
    \includegraphics[scale=0.6]{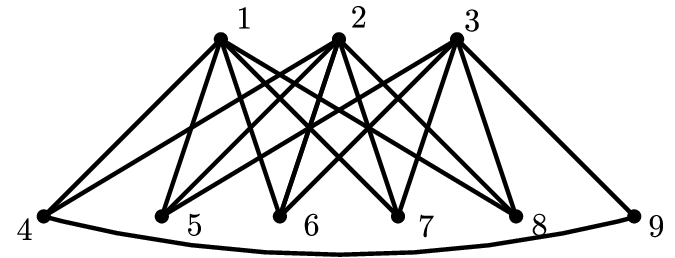}
    \caption{A $1$-dimensional flag simplicial complex.}
    \label{fig: 1-dimensional flag simplicial complex}
\end{figure}

By Corollary~\ref{coro: lin related}, $I_{\Delta}^{[2]}$ is not linearly related, since, for instance, the induced subgraph of $\Delta$ on the vertex set $\{1,2,4,5,6,7\}$ is isomorphic to $K_{2,4}$. Using \textit{Macaulay2} \cite{M2}, the Betti table of $I_{\Delta}^{[2]}$ over a field of characteristic zero is the following:
\[
\begin{array}{r|cccccc}
      & 0 & 1 & 2 & 3 & 4 & 5 \\
\hline
\mathrm{total:} & 91 & 336 & 516 & 406 & 162 & 26 \\
0: & . & . & . & . & . & . \\
1: & . & . & . & . & . & . \\
2: & . & . & . & . & . & . \\
3: & . & . & . & . & . & . \\
4: & 91 & 329 & 497 & 389 & 157 & 26 \\
5: & . & 7 & 19 & 17 & 5 & .
\end{array}
\]

Note that $\beta_{1,6}$ coincides with the number of induced subgraphs of $\Delta$ that are isomorphic to $K_{2,4}$. In this case, there are exactly $7$ such subgraphs, induced by the vertex sets $\{1,2,4,5,6,7\}$, $\{1,2,4,5,6,8\}$, $\{1,2,4,5,7,8\}$, $\{1,2,4,6,7,8\}$, $\{1,2,5,6,7,8\}$, $\{2,3,5,6,7,8\}$, and $\{1,3,5,6,7,8\}$.

This phenomenon is not restricted to the present example, but holds in general: in the next theorem we express $\beta_{1,2p+2}(I_{\Delta}^{[p]})$, which measures the failure of the linearly related property of $I_{\Delta}^{[p]}$, in terms of the number of induced subgraphs isomorphic to $K_{2+i,\,2p-i}$.
\end{Example}

\begin{Theorem}\label{Thm: betti 2,2p+2}
Let $\Delta$ be a $1$-dimensional flag simplicial complex on $[n]$ and let $2 \leq p \leq \nu(\overline{\Delta})$. Then
\[
\beta_{1,2p+2}(I_{\Delta}^{[p]})
=
\sum_{\substack{0 \leq i \leq p-1 \\ i \text{ even}}} \big\vert\{\text{induced subgraphs of } \Delta \text{ isomorphic to } K_{2+i,\,2p-i}\}\big\vert.
\]
\end{Theorem}

The graphs $K_{2+i,\,2p-i}$ with $i$ even and $0\leq i\leq p-1$ are pairwise non-isomorphic, and they are, up to isomorphism, all the graphs $K_{2+i,\,2p-i}$ with $i$ even and $0\leq i\leq p$: if $p$ is even, then $K_{2+p,\,p}\cong K_{p,\,p+2}$ is already obtained for $i=p-2$.

\begin{proof}
By Hochster's formula,
\[
\beta_{1,2p+2}(I_{\Delta}^{[p]})=\sum_{W\subseteq[n],\ |W|=2p+2}\dim_K\widetilde{H}_{2p-1}(\Delta_W^{[p]};K).
\]
Fix $W\subseteq[n]$ with $|W|=2p+2$. As in the proof of Proposition~\ref{Thm: dimension}(1), $\overline{\Delta}_W$ has a $p$-matching, so $\Delta_W$ satisfies the assumptions of Corollary~\ref{coro: lin related}, and $\dim_K\widetilde{H}_{2p-1}(\Delta_W^{[p]};K)=\beta_{1,2p+2}(I_{\Delta_W}^{[p]})$ by Hochster's formula applied to $\Delta_W$. If $\Delta_W$ is not isomorphic to some $K_{2+i,\,2p-i}$ with $i$ even and $0\leq i\leq p$, then $\Delta_W$, having only $2p+2$ vertices, contains no such induced subgraph, so $I_{\Delta_W}^{[p]}$ is linearly related by Corollary~\ref{coro: lin related} and $\widetilde{H}_{2p-1}(\Delta_W^{[p]};K)=0$. It remains to show that $\dim_K\widetilde{H}_{2p-1}(\Delta_W^{[p]};K)=1$ when $\Delta_W\cong K_{2+i,\,2p-i}$.

Assume $\Delta_W\cong K_{2+i,\,2p-i}$, where, as explained after the statement, we may choose $i$ even with $0\leq i\leq p-1$, so that $2+i\leq 2p-i$, and relabel the vertices so that $W=A\sqcup B$ with $A=\{1,\dots,i+2\}$, $B=\{i+3,\dots,2p+2\}$ and edge set $\{\{h,k\}:h\in A,\ k\in B\}$. A complete bipartite graph $K_{j,\,2p-j}$ with $j$ odd and $1\leq j\leq p$ is an induced subgraph of $K_{2+i,\,2p-i}$ if and only if $j\leq i+2$ and $2p-j\leq 2p-i$, that is, $j=i+1$. Hence, by Theorem~\ref{Thm: Facet, pure, CM}(1), the $(2p-1)$-dimensional faces of $\Delta_W^{[p]}$ are the sets
\[
F_{h,k}=W\setminus\{h,k\},\qquad h\in A,\ k\in B.
\]
Since $\Delta_W^{[p]}$ has no faces of dimension $2p$ by Proposition~\ref{Thm: dimension}, we have $\widetilde{H}_{2p-1}(\Delta_W^{[p]};K)=\ker\delta_{2p-1}$. Let
\[
\Lambda=\sum_{h\in A}\sum_{k\in B}\lambda_{h,k}\,[F_{h,k}],\qquad \lambda_{h,k}\in K.
\]
Every face occurring in $\delta_{2p-1}([F_{h,k}])$ is of the form $W\setminus\{h_1,h_2,k\}$ with $h_1<h_2$ in $A$, or $W\setminus\{h,k_1,k_2\}$ with $k_1<k_2$ in $B$, and $W\setminus\{h_1,h_2,k\}$ (resp. $W\setminus\{h,k_1,k_2\}$) is contained in exactly two $(2p-1)$-dimensional faces, namely $F_{h_1,k}$ and $F_{h_2,k}$ (resp. $F_{h,k_1}$ and $F_{h,k_2}$). In $F_{h_2,k}$ the vertex $h_1$ occupies the position $h_1$, while in $F_{h_1,k}$ the vertex $h_2$ occupies the position $h_2-1$; similarly, in $F_{h,k_2}$ the vertex $k_1$ occupies the position $k_1-1$, while in $F_{h,k_1}$ the vertex $k_2$ occupies the position $k_2-2$. Therefore the coefficients of $[W\setminus\{h_1,h_2,k\}]$ and of $[W\setminus\{h,k_1,k_2\}]$ in $\delta_{2p-1}(\Lambda)$ are
\[
(-1)^{h_1+1}\lambda_{h_2,k}+(-1)^{h_2}\lambda_{h_1,k}\qquad\text{and}\qquad (-1)^{k_1}\lambda_{h,k_2}+(-1)^{k_2+1}\lambda_{h,k_1},
\]
respectively. Consequently, $\Lambda$ is a cycle if and only if $(-1)^{h_1}\lambda_{h_1,k}=(-1)^{h_2}\lambda_{h_2,k}$ and $(-1)^{k_1}\lambda_{h,k_1}=(-1)^{k_2}\lambda_{h,k_2}$ for all $h,h_1,h_2\in A$ and $k,k_1,k_2\in B$, that is, if and only if $(-1)^{h+k}\lambda_{h,k}$ is independent of $h$ and $k$. Hence
\[
\widetilde{H}_{2p-1}(\Delta_W^{[p]};K)=\ker\delta_{2p-1}=K\,\Gamma,\qquad \Gamma=\sum_{h=1}^{i+2}\sum_{k=i+3}^{2p+2}(-1)^{h+k}\,[F_{h,k}],
\]
and $\dim_K\widetilde{H}_{2p-1}(\Delta_W^{[p]};K)=1$.
\end{proof}

\section{Linear resolutions and the neighbourhood complex}\label{sec: linearity simpl complex}

In this section we characterize when $I_\Delta^{[p]}$ has a linear resolution. We show that the nonlinear syzygies in higher homological degrees are controlled by the topology of the neighbourhood complex introduced by Lov\'asz \cite{Lovasz1978}.

\subsection{Top-dimensional homology}

Let $\Delta$ be a $1$-dimensional flag simplicial complex on $[n]$ and let $p\geq2$. 
For every $p\geq2$, every subset of $[n]$ with at most $2p-1$ elements is a face of $\Delta^{[p]}$, no subset with $2p+1$ elements is a face, and the faces with $2p$ elements, which we call the \textit{top faces} of $\Delta^{[p]}$, are the sets in $\mathcal{F}(\Delta,K_{a,b})$ with $a,b$ odd and $a+b=2p$. In particular $C_{2p}(\Delta^{[p]})=0$, so $\widetilde{H}_{2p-1}(\Delta^{[p]};K)=\ker\delta_{2p-1}$ is the space of \textit{top cycles} of $\Delta^{[p]}$. Since the chains of $\Delta_W^{[p]}=(\Delta^{[p]})_W$ are the chains of $\Delta^{[p]}$ supported on $W$, we obtain
\begin{equation}\tag{\#}\label{Eq: injectivity}
\widetilde{H}_{2p-1}(\Delta_W^{[p]};K)\subseteq\widetilde{H}_{2p-1}(\Delta^{[p]};K)\qquad\text{for every }W\subseteq[n].
\end{equation}

\begin{Lemma}\label{Lem: linear iff top homology}
Let $\Delta$ be a $1$-dimensional flag simplicial complex on $[n]$ and let $2\leq p\leq\nu(\overline{\Delta})$.
\begin{enumerate}
\item $\beta_{i,j}(I_\Delta^{[p]})=0$ unless $j-i\in\{2p,2p+1\}$, and for every $i\geq0$
\[
\beta_{i,i+2p+1}(I_\Delta^{[p]})=\sum_{W\subseteq[n],\ |W|=i+2p+1}\dim_K\widetilde{H}_{2p-1}(\Delta_W^{[p]};K).
\]
\item $I_\Delta^{[p]}$ has a linear resolution if and only if $\widetilde{H}_{2p-1}(\Delta^{[p]};K)=0$.
\item If $\beta_{j,j+2p+1}(I_\Delta^{[p]})\neq0$, then $\beta_{j',j'+2p+1}(I_\Delta^{[p]})\neq0$ for every $j\leq j'\leq n-2p-1$.
\end{enumerate}
\end{Lemma}

\begin{proof}
For every $W\subseteq[n]$, the complex $\Delta_W^{[p]}$ contains the $(2p-2)$-skeleton of the simplex on $W$ and has dimension at most $2p-1$, so its reduced homology vanishes outside the degrees $2p-2$ and $2p-1$. Hochster's formula gives (1). By (1), $I_\Delta^{[p]}$ has a linear resolution if and only if $\widetilde{H}_{2p-1}(\Delta_W^{[p]};K)=0$ for all $W$, and by \eqref{Eq: injectivity} this happens if and only if it holds for $W=[n]$. For (3), if $\widetilde{H}_{2p-1}(\Delta_W^{[p]};K)\neq0$ with $|W|=j+2p+1$, then $\widetilde{H}_{2p-1}(\Delta_{W'}^{[p]};K)\neq0$ for every $W'\supseteq W$ by \eqref{Eq: injectivity} applied to $\Delta_{W'}$.
\end{proof}

\subsection{The neighbourhood complex and the main results}

For a vertex $v$ of a graph $\Delta$ we denote by $N_\Delta(v)$ the set of neighbours of $v$, and for a nonempty set $A$ of vertices we put $N_\Delta(A)=\bigcap_{a\in A}N_\Delta(a)$. Two nonempty sets $A,B\subseteq[n]$ form a \textit{complete pair} if every vertex of $A$ is adjacent to every vertex of $B$. If $\Delta$ is triangle-free and $(A,B)$ is a complete pair, then $A\cap B=\emptyset$, both $A$ and $B$ are independent, and $\Delta_{A\cup B}\cong K_{|A|,|B|}$; since this graph is connected, its bipartition $\{A,B\}$ is determined by $A\cup B$.

\begin{Definition}\label{Def: neighbourhood complex}
The \textit{neighbourhood complex} of a graph $\Delta$ on $[n]$ is
\[
\Nb(\Delta)=\bigl\langle N_\Delta(v):v\in[n]\bigr\rangle=\{F\subseteq[n]:F\subseteq N_\Delta(v)\text{ for some }v\in[n]\},
\]
that is, the simplicial complex whose faces are the sets of vertices having a common neighbour. If $\Delta$ is bipartite with bipartition $[n]=V_1\sqcup V_2$, then every face of $\Nb(\Delta)$ is contained in $V_1$ or in $V_2$, and $\Nb(\Delta)$ is the disjoint union of $\Nb_1(\Delta)=\langle N_\Delta(v):v\in V_2\rangle$ and $\Nb_2(\Delta)=\langle N_\Delta(v):v\in V_1\rangle$.
\end{Definition}


\begin{Definition}[Condition B$_p$]\label{Def: Bp}
Let $p\geq2$. We say that a $1$-dimensional flag simplicial complex $\Delta$ satisfies \textit{condition} \Bp\ if $\Delta$ has no induced subgraph isomorphic to $K_{2+i,\,2p-i}$ for any even integer $0\leq i\leq p$.
\end{Definition}

\begin{Remark}\label{Rem: Bp}
Let $\Delta$ be a $1$-dimensional flag simplicial complex and let $p\geq2$.
\begin{enumerate}
\item If $p\leq\nu(\overline{\Delta})$, then $\Delta$ satisfies \Bp\ if and only if $I_\Delta^{[p]}$ is linearly related, by Corollary~\ref{coro: lin related}.
\item Condition \Bp\ is inherited by the induced subgraphs of $\Delta$.
\item Condition \Bp\ implies condition (B$_{p+1}$). Indeed, a complete bipartite graph on $2p+4$ vertices whose parts have even sizes at least two contains, after deleting two vertices from a part of size at least four, an induced complete bipartite graph on $2p+2$ vertices with the same property.
\end{enumerate}
\end{Remark}

\begin{Theorem}\label{Thm: main criterion}
Let $\Delta$ be a $1$-dimensional flag simplicial complex and let $2\leq p\leq\nu(\overline{\Delta})$. The following conditions are equivalent:
\begin{enumerate}
\item $I_\Delta^{[p]}$ has a linear resolution over $K$;
\item $I_\Delta^{[p]}$ is linearly related and $\widetilde{H}_{2p-1}(\Nb(\Delta);K)=0$.
\end{enumerate}
\end{Theorem}

Under condition \Bp\ we can also compute the dimension of the space of top cycles, that is, by Lemma~\ref{Lem: linear iff top homology}(1), the Betti number $\beta_{n-2p-1,n}(I_\Delta^{[p]})=\dim_K\widetilde{H}_{2p-1}(\Delta^{[p]};K)$.

\begin{Theorem}\label{Thm: dimension formula}
Let $\Delta$ be a $1$-dimensional flag simplicial complex satisfying \Bp\ for some $p\geq2$, and put $h=\dim_K\widetilde{H}_{2p-1}(\Nb(\Delta);K)$.
\begin{enumerate}
\item If $p$ is even, then $\dim_K\widetilde{H}_{2p-1}(\Delta^{[p]};K)=\frac p2\,h$.
\item If $p$ is odd, then $\frac{p-1}{2}\,h\leq\dim_K\widetilde{H}_{2p-1}(\Delta^{[p]};K)\leq\frac{p+1}{2}\,h$.
\item If $\Delta$ is bipartite with bipartition $V_1\sqcup V_2$, then $\widetilde{H}_{2p-1}(\Nb_1(\Delta);K)$ and $\widetilde{H}_{2p-1}(\Nb_2(\Delta);K)$ both have dimension $h/2$, and $\dim_K\widetilde{H}_{2p-1}(\Delta^{[p]};K)=\frac p2\,h$.
\end{enumerate}
\end{Theorem}

\begin{Remark}\label{Rem: Kahle}
Let $\Delta$ satisfy \Bp. Since $\Delta$ is triangle-free, every subgraph of $\Delta$ isomorphic to a complete bipartite graph $K_{a,b}$ with $a,b\geq1$ is induced, so $\Delta$ has no subgraph isomorphic to $K_{2,2p}$. By a lemma of Kahle \cite[Lemma~2.3]{Kahle2007}, $\Nb(\Delta)$ strong deformation retracts onto a simplicial complex of dimension at most $2p-1$. Consequently $\widetilde{H}_k(\Nb(\Delta);K)=0$ for every $k\geq2p$, the group $H_{2p-1}(\Nb(\Delta);\ZZ)$ is free, and by the universal coefficient theorem \cite[Theorem~3A.3]{hatcher}
\[
\widetilde{H}_{2p-1}(\Nb(\Delta);K)\cong\bigl(H_{2p-1}(\Nb(\Delta);\ZZ)\otimes K\bigr)\oplus\Tor_1^{\ZZ}\bigl(H_{2p-2}(\Nb(\Delta);\ZZ),K\bigr).
\]
Hence $\widetilde{H}_{2p-1}(\Nb(\Delta);K)$ can depend on the field only through the torsion of $H_{2p-2}(\Nb(\Delta);\ZZ)$, in accordance with Proposition~\ref{Prop: exceptional characteristics}. However, Kahle's bound does not tell whether $\widetilde{H}_{2p-1}(\Nb(\Delta);K)$ vanishes: we will see that the crown graphs and the graphs constructed in Section~\ref{Sec: characteristic} satisfy \Bp\ and have nonzero top homology.
\end{Remark}

The crown graph shows that the condition on $\Nb(\Delta)$ in Theorem~\ref{Thm: main criterion} is not implied by \Bp. The following corollary shows that an induced crown always prevents $I_\Delta^{[p]}$ from having a linear resolution.

\begin{Corollary}\label{Coro: crown}
Let $p\geq2$. Then $Cr(2p+1)$ satisfies \Bp\ and $\dim_K\widetilde{H}_{2p-1}(Cr(2p+1)^{[p]};K)=p$ for every field $K$. Consequently, if a $1$-dimensional flag simplicial complex $\Delta$ with $p\leq\nu(\overline{\Delta})$ contains an induced subgraph isomorphic to $Cr(2p+1)$, then $I_\Delta^{[p]}$ does not have a linear resolution.
\end{Corollary}

\begin{proof}
Let $\{u_1,\dots,u_{2p+1}\}\sqcup\{v_1,\dots,v_{2p+1}\}$ be the bipartition of $Cr(2p+1)$. A pair $(\{u_i:i\in I\},\{v_j:j\in J\})$ is complete if and only if $I\cap J=\emptyset$, so every induced complete bipartite subgraph has at most $2p+1$ vertices, and \Bp\ holds. The facets of $\Nb_2(Cr(2p+1))$ are the sets $\{v_j:j\neq i\}$, $i\in[2p+1]$, so $\Nb_2(Cr(2p+1))$ is the boundary of the simplex on $\{v_1,\dots,v_{2p+1}\}$, a $(2p-1)$-sphere, and $\dim_K\widetilde{H}_{2p-1}(\Nb_2(Cr(2p+1));K)=1$. Theorem~\ref{Thm: dimension formula}(3) gives the first assertion. The second follows from \eqref{Eq: injectivity} and Lemma~\ref{Lem: linear iff top homology}(2).
\end{proof}

\subsection{A dimension-shifting lemma}

The proofs of Theorems~\ref{Thm: main criterion} and~\ref{Thm: dimension formula} rely on a double complex attached to a relation. We use the convention of Section~\ref{SubSec2} for oriented simplices with vertices listed in an arbitrary order; for disjoint sets $A$ and $B$ we write $[A]$ for the oriented simplex on $A$ with the vertices in increasing order and $[A,B]$ for the concatenation of $[A]$ and $[B]$. Then
\begin{equation}\tag{*}\label{Eq: boundary of join}
\delta[A,B]=[\delta A,B]+(-1)^{|A|}[A,\delta B]\qquad\text{and}\qquad[B,A]=(-1)^{|A|\,|B|}[A,B],
\end{equation}
where, for $A=\{a_1<\dots<a_k\}$, we put $[\delta A,B]=\sum_{i=1}^k(-1)^{i+1}[A\setminus\{a_i\},B]$, and similarly for $[A,\delta B]$.

For a chain $c$ of a simplicial complex and a face $G$, we can write uniquely $c=[G,c_G]+c'$, where $c_G$ is a linear combination of simplices disjoint from $G$ and no simplex occurring in $c'$ contains $G$; we call $c_G$ the \textit{link} of $G$ in $c$. By \eqref{Eq: boundary of join}, the part of $\delta c$ formed by the simplices containing $G$ is $(-1)^{|G|}[G,\delta c_G]$; hence, if $c$ is a cycle, then $c_G$ is a cycle, and it is nonzero if and only if $G$ is contained in a simplex occurring in $c$.

Let $\rho\subseteq X\times Y$ be a relation between finite totally ordered sets. For nonempty $A\subseteq X$ and $B\subseteq Y$ let $\rho(A)=\{y\in Y:(x,y)\in\rho\text{ for all }x\in A\}$ and $\rho^{-1}(B)=\{x\in X:(x,y)\in\rho\text{ for all }y\in B\}$, and consider the simplicial complexes
\[
\Nb_X(\rho)=\{A\subseteq X:A=\emptyset\text{ or }\rho(A)\neq\emptyset\},\qquad\Nb_Y(\rho)=\{B\subseteq Y:B=\emptyset\text{ or }\rho^{-1}(B)\neq\emptyset\}.
\]
A pair $(A,B)$ with $A\subseteq X$ and $B\subseteq Y$ is \textit{admissible} if either $A,B\neq\emptyset$ and $A\times B\subseteq\rho$, or one of $A,B$ is empty and the other one is a face of $\Nb_X(\rho)$ or of $\Nb_Y(\rho)$, respectively. For $a,b\geq0$, let $\cC_{a,b}=\cC_{a,b}(\rho)$ be the $K$-vector space with basis the symbols $[A]\otimes[B]$, where $(A,B)$ is admissible, $|A|=a$ and $|B|=b$, and define
\[
\delta'([A]\otimes[B])=[\delta A]\otimes[B],\qquad\delta''([A]\otimes[B])=[A]\otimes[\delta B],
\]
with $\delta'=0$ on $\cC_{0,b}$ and $\delta''=0$ on $\cC_{a,0}$. Removing an element from $A$ or from $B$ preserves admissibility; hence $\delta'$ and $\delta''$ are well defined, $\delta'^2=\delta''^2=0$ and $\delta'\delta''=\delta''\delta'$. Moreover, $(\cC_{0,\bullet},\delta'')$ is the augmented chain complex of $\Nb_Y(\rho)$, with $\cC_{0,b}=C_{b-1}(\Nb_Y(\rho))$, and symmetrically for $(\cC_{\bullet,0},\delta')$.

\begin{Lemma}\label{Lem: exact rows}
If $b\geq1$, then $\ker\delta'\cap\cC_{a,b}=\delta'(\cC_{a+1,b})$ for every $a\geq0$. If $a\geq1$, then $\ker\delta''\cap\cC_{a,b}=\delta''(\cC_{a,b+1})$ for every $b\geq0$.
\end{Lemma}

\begin{proof}
Since $\delta'$ does not change the second factor, $(\cC_{\bullet,b},\delta')$ is the direct sum, over the sets $B\subseteq Y$ with $|B|=b$, of the complexes spanned by the elements $[A]\otimes[B]$. If $\rho^{-1}(B)=\emptyset$, this summand is zero; otherwise $A$ runs over all subsets of $\rho^{-1}(B)$, including $A=\emptyset$, and the summand is isomorphic to the augmented chain complex of the simplex on $\rho^{-1}(B)$, which is exact. The second assertion is symmetric.
\end{proof}

For $a,b\geq0$ we put $\cZ_{a,b}=\ker\delta'\cap\ker\delta''\cap\cC_{a,b}$ and $\cQ_{a,b}=\cZ_{a,b}/\delta'\delta''(\cC_{a+1,b+1})$.

\begin{Lemma}[Dimension shifting]\label{Lem: dimension shifting}
Let $a,b\geq1$.
\begin{enumerate}
\item There are isomorphisms $\cQ_{a,b}\cong\cQ_{a-1,b+1}\cong\cdots\cong\cQ_{0,a+b}\cong\widetilde{H}_{a+b-1}(\Nb_Y(\rho);K)$, and symmetrically $\cQ_{a,b}\cong\widetilde{H}_{a+b-1}(\Nb_X(\rho);K)$.
\item If $\cC_{a+1,b+1}=0$, then $\cZ_{a,b}\cong\widetilde{H}_{a+b-1}(\Nb_Y(\rho);K)$.
\item The isomorphisms in \textup{(1)} commute with the maps induced by the inclusions $\cC_{\bullet,\bullet}(\rho')\subseteq\cC_{\bullet,\bullet}(\rho)$ and $\Nb_Y(\rho')\subseteq\Nb_Y(\rho)$, where $\rho'=\rho\cap(X'\times Y')$ for some $X'\subseteq X$ and $Y'\subseteq Y$.
\end{enumerate}
\end{Lemma}

\begin{proof}
(1) Let $z\in\cZ_{a,b}$. By Lemma~\ref{Lem: exact rows} there exists $u\in\cC_{a,b+1}$ with $\delta''u=z$, and $\delta'u\in\cZ_{a-1,b+1}$ because $\delta''\delta'u=\delta'z=0$. Two choices of $u$ differ by an element $\delta''t$ with $t\in\cC_{a,b+2}$, which changes $\delta'u$ by $\delta'\delta''t$; if $z$ is replaced by $z+\delta'\delta''w$, we may replace $u$ by $u+\delta'w$, which does not change $\delta'u$. Hence $\varphi([z])=[\delta'u]$ is a well-defined linear map $\cQ_{a,b}\to\cQ_{a-1,b+1}$. Symmetrically, for $z'\in\cZ_{a-1,b+1}$ we choose $v\in\cC_{a,b+1}$ with $\delta'v=z'$, which is possible by Lemma~\ref{Lem: exact rows} since $b+1\geq1$, and put $\psi([z'])=[\delta''v]$. If $z=\delta''u$, then $u$ is an admissible choice of $v$ for $z'=\delta'u$, so $\psi\varphi=\mathrm{id}$, and similarly $\varphi\psi=\mathrm{id}$. Iterating, $\cQ_{a,b}\cong\cQ_{0,a+b}$. Finally, $\cZ_{0,m}$ is the space of reduced $(m-1)$-cycles of $\Nb_Y(\rho)$, and $\delta'\colon\cC_{1,m+1}\to\cC_{0,m+1}$ is surjective: for a face $B$ of $\Nb_Y(\rho)$ with $|B|=m+1$ choose $x\in\rho^{-1}(B)$, so that $\delta'([x]\otimes[B])=[\emptyset]\otimes[B]$. Hence $\delta'\delta''(\cC_{1,m+1})=\delta''\delta'(\cC_{1,m+1})=\delta''(\cC_{0,m+1})$ is the space of boundaries, and $\cQ_{0,m}=\widetilde{H}_{m-1}(\Nb_Y(\rho);K)$. The statement for $X$ is obtained by exchanging the roles of $X$ and $Y$. Assertion (2) is clear from the definition of $\cQ_{a,b}$, and (3) holds because a lift chosen for $\rho'$ is also a lift for $\rho$.
\end{proof}

\begin{Remark}\label{Rem: Dowker}
For every relation $\rho\subseteq X\times Y$ and every $k\geq1$, Lemma~\ref{Lem: dimension shifting}(1) gives an isomorphism $\widetilde{H}_k(\Nb_X(\rho);K)\cong\widetilde{H}_k(\Nb_Y(\rho);K)$. This is the homological counterpart of Dowker's theorem \cite{Dowker1952}, which asserts that the two complexes are homotopy equivalent. In particular, for a bipartite graph $\Delta$ with bipartition $V_1\sqcup V_2$ and $\rho=E(\Delta)\subseteq V_1\times V_2$, we get $\widetilde{H}_k(\Nb_1(\Delta);K)\cong\widetilde{H}_k(\Nb_2(\Delta);K)$ for $k\geq1$.
\end{Remark}

\subsection{Proofs of Theorems~\ref{Thm: main criterion} and~\ref{Thm: dimension formula}}\label{sub:proof theorem}

Let $\Delta$ be a $1$-dimensional flag simplicial complex on $[n]$ and $p\geq2$. For odd integers $1\leq a\leq b$ with $a+b=2p$, let $U_{a,b}\subseteq C_{2p-1}(\Delta^{[p]})$ be the subspace spanned by the top faces in $\mathcal{F}(\Delta,K_{a,b})$, so that $C_{2p-1}(\Delta^{[p]})=\bigoplus_{a\leq b}U_{a,b}$.

\begin{Lemma}\label{Lem: faces of top faces}
Let $W\in\mathcal{F}(\Delta,K_{a,b})$ be a top face of $\Delta^{[p]}$ and let $F=W\setminus\{x\}$ for some $x\in W$. Then either $F$ is independent, which happens exactly when $a=1$ and $x$ is the centre of the star $\Delta_W$, or $\Delta_F\cong K_{c,d}$ with $c\geq2$ even and $d$ odd. Let $o(F)$ be $F$ in the first case and the part of size $d$ in the second one, and let $e(F)=F\setminus o(F)$. Then $\{a,b\}=\{|e(F)|+1,|o(F)|\}$.
\end{Lemma}

\begin{proof}
Let $W=A\sqcup B$ be the bipartition of $\Delta_W$ with $|A|=a$, $|B|=b$. If $x\in A$, then $F=(A\setminus\{x\})\sqcup B$; this is the independent set $B$ if $a=1$, and otherwise $(A\setminus\{x\},B)$ is a complete pair with parts of even size $a-1$ and odd size $b$. The case $x\in B$ is symmetric.
\end{proof}

By Lemma~\ref{Lem: faces of top faces}, the type $\{a,b\}$ of a top face can be read off from each face of its boundary. Hence the boundaries of elements of different subspaces $U_{a,b}$ have disjoint supports, and
\begin{equation}\tag{I}\label{Eq: splitting}
\widetilde{H}_{2p-1}(\Delta^{[p]};K)=\bigoplus_{a\leq b}\bigl(\widetilde{H}_{2p-1}(\Delta^{[p]};K)\cap U_{a,b}\bigr).
\end{equation}
We apply the constructions of the previous subsection to the adjacency relation $\rho=\{(u,v)\in[n]\times[n]:\{u,v\}\in\Delta\}$, for which $\Nb_X(\rho)=\Nb_Y(\rho)=\Nb(\Delta)$. For $a,b\geq1$ the admissible pairs $(A,B)$ with $|A|=a$ and $|B|=b$ are exactly the ordered complete pairs, hence
\begin{equation}\tag{II}\label{Eq: Cab}
\cC_{a,b}(\rho)=0\quad\text{if and only if}\quad\Delta\text{ has no induced subgraph isomorphic to }K_{a,b}.
\end{equation}
In particular, \Bp\ holds if and only if $\cC_{a+1,b+1}(\rho)=0$ for all odd $a,b$ with $a+b=2p$.

\begin{Proposition}\label{Prop: unequal class}
Let $a<b$ be odd integers with $a+b=2p$. The linear map $J\colon\cC_{a,b}(\rho)\to U_{a,b}$, $[A]\otimes[B]\mapsto[A,B]$, is an isomorphism and restricts to an isomorphism $\cZ_{a,b}\cong\widetilde{H}_{2p-1}(\Delta^{[p]};K)\cap U_{a,b}$.
\end{Proposition}

\begin{proof}
Each $W\in\mathcal{F}(\Delta,K_{a,b})$ has a unique decomposition $W=A\sqcup B$ into parts with $|A|=a<b=|B|$, so $J$ maps a basis bijectively onto a basis, up to signs. Let $J'\colon\cC_{a-1,b}\oplus\cC_{a,b-1}\to C_{2p-2}(\Delta^{[p]})$ be defined by the same rule. By \eqref{Eq: boundary of join}, $\delta J(z)=J'(\delta'z+(-1)^a\delta''z)$ for $z\in\cC_{a,b}$. A basis element of $\cC_{a-1,b}$ is mapped to $\pm[F]$ with $o(F)=B$ and $|o(F)|=b$, and a basis element of $\cC_{a,b-1}$ to $\pm[F]$ with $o(F)=A$ and $|o(F)|=a\neq b$. By Lemma~\ref{Lem: faces of top faces} these faces are pairwise distinct, so $J'$ is injective, and $\delta J(z)=0$ if and only if $\delta'z=0$ and $\delta''z=0$.
\end{proof}

\begin{Proposition}\label{Prop: equal class}
Let $p$ be odd. Define $\tau\colon\cC_{p,p}(\rho)\to\cC_{p,p}(\rho)$ by $\tau([A]\otimes[B])=-[B]\otimes[A]$ and $L\colon U_{p,p}\to\cC_{p,p}(\rho)$ by $L([A,B])=[A]\otimes[B]-[B]\otimes[A]$. Then $L$ is injective in every characteristic, its image is the space of $\tau$-invariant vectors, and $L$ restricts to an isomorphism $\widetilde{H}_{2p-1}(\Delta^{[p]};K)\cap U_{p,p}\cong\cZ_{p,p}^\tau=\{z\in\cZ_{p,p}:\tau(z)=z\}$.
\end{Proposition}

\begin{proof}
Since $[B,A]=(-1)^{p^2}[A,B]=-[A,B]$ and $L([B,A])=-L([A,B])$, the map $L$ is well defined. Distinct top faces of this type involve disjoint pairs $\{[A]\otimes[B],[B]\otimes[A]\}$ of basis vectors, and $A\neq B$, so $L$ is injective, and its image is clearly the space of $\tau$-invariant vectors. Let $\kappa\colon\cC_{p-1,p}(\rho)\to C_{2p-2}(\Delta^{[p]})$ be defined by $[A']\otimes[B]\mapsto[A',B]$; it is injective by Lemma~\ref{Lem: faces of top faces}. For $c=[A,B]$ we have $\delta c=[\delta A,B]-[A,\delta B]$ by \eqref{Eq: boundary of join}, and $\kappa\delta'L(c)=[\delta A,B]-[\delta B,A]$; since $p(p-1)$ is even, $[\delta B,A]=[A,\delta B]$. Therefore $\delta c=\kappa\delta'L(c)$ for every $c\in U_{p,p}$, and $c$ is a cycle if and only if $\delta'L(c)=0$. Finally, $\delta''\tau=-\tau'\delta'$, where $\tau'([A']\otimes[B])=[B]\otimes[A']$; hence, for a $\tau$-invariant vector $x$, $\delta'x=0$ implies $\delta''x=-\tau'\delta'x=0$.
\end{proof}

\begin{proof}[Proof of Theorem~\ref{Thm: main criterion}]
Since $p\geq2$, we have $1<2p-1$. By Proposition~\ref{Prop: unequal class} and Lemma~\ref{Lem: dimension shifting}(1), the space $\widetilde{H}_{2p-1}(\Delta^{[p]};K)\cap U_{1,2p-1}\cong\cZ_{1,2p-1}$ surjects onto $\cQ_{1,2p-1}\cong\widetilde{H}_{2p-1}(\Nb(\Delta);K)$. Hence, if $\widetilde{H}_{2p-1}(\Nb(\Delta);K)\neq0$, then $I_\Delta^{[p]}$ does not have a linear resolution by Lemma~\ref{Lem: linear iff top homology}(2). Together with Corollary~\ref{coro: lin related}, this proves (1)$\Rightarrow$(2).

Conversely, assume (2). By Corollary~\ref{coro: lin related}, $\Delta$ satisfies \Bp, hence $\cC_{a+1,b+1}(\rho)=0$ for all odd $a,b$ with $a+b=2p$ by \eqref{Eq: Cab}, and Lemma~\ref{Lem: dimension shifting}(2) gives $\cZ_{a,b}\cong\widetilde{H}_{2p-1}(\Nb(\Delta);K)=0$. By Propositions~\ref{Prop: unequal class} and~\ref{Prop: equal class}, no subspace $U_{a,b}$ contains a nonzero cycle, and \eqref{Eq: splitting} gives $\widetilde{H}_{2p-1}(\Delta^{[p]};K)=0$. We conclude by Lemma~\ref{Lem: linear iff top homology}(2).
\end{proof}

\begin{proof}[Proof of Theorem~\ref{Thm: dimension formula}]
By \eqref{Eq: Cab}, Lemma~\ref{Lem: dimension shifting}(2) and Propositions~\ref{Prop: unequal class} and~\ref{Prop: equal class}, each of the $\lfloor p/2\rfloor$ types $\{a,b\}$ with $a<b$ contributes $\dim_K\cZ_{a,b}=h$ to \eqref{Eq: splitting}, and, if $p$ is odd, the type $\{p,p\}$ contributes $\dim_K\cZ_{p,p}^\tau\leq\dim_K\cZ_{p,p}=h$. This proves (1) and (2).

For (3), every top face $W$ of $\Delta^{[p]}$ is a complete pair $(W\cap V_1,W\cap V_2)$ with $k=|W\cap V_1|$ odd and $1\leq k\leq2p-1$. Let $V_k\subseteq C_{2p-1}(\Delta^{[p]})$ be the subspace spanned by the top faces with $|W\cap V_1|=k$. A face in the boundary of such a $W$ has $k-1$ or $k$ vertices in $V_1$, and the sets $\{k-1,k\}$ are pairwise disjoint for distinct odd $k$; hence the top cycles split as a direct sum over $k$. Applying the constructions of the previous subsection to the relation $\rho_{12}=E(\Delta)\subseteq V_1\times V_2$ and arguing as in the proof of Proposition~\ref{Prop: unequal class}, now using that a face $F$ is determined by the pair $(F\cap V_1,F\cap V_2)$, we see that the map $[A]\otimes[B]\mapsto[A,B]$ identifies $\cZ_{k,2p-k}(\rho_{12})$ with the space of cycles in $V_k$. By \Bp, \eqref{Eq: Cab} and Lemma~\ref{Lem: dimension shifting}, $\cZ_{k,2p-k}(\rho_{12})\cong\widetilde{H}_{2p-1}(\Nb_2(\Delta);K)\cong\widetilde{H}_{2p-1}(\Nb_1(\Delta);K)$. Since $\Nb(\Delta)=\Nb_1(\Delta)\sqcup\Nb_2(\Delta)$ and $2p-1\geq1$, the number $h$ is the sum of these two equal dimensions, so each of them equals $h/2$, and summing over the $p$ odd values of $k$ gives $\dim_K\widetilde{H}_{2p-1}(\Delta^{[p]};K)=p\cdot h/2$.
\end{proof}

The case $(a,b)=(1,2p-1)$ admits the following concrete description, which we shall use repeatedly. A basis element $[v]\otimes[B]$ of $\cC_{1,2p-1}(\rho)$ is a vertex $v$ together with a $(2p-1)$-subset $B$ of $N_\Delta(v)$, so an element of $\cC_{1,2p-1}(\rho)$ is a family $(\sigma_v)_{v\in[n]}$ of $(2p-2)$-chains of the simplices $\langle N_\Delta(v)\rangle$. The condition $\delta''=0$ says that every $\sigma_v$ is a cycle, and the condition $\delta'=0$ says that $\sum_v\sigma_v=0$ in $C_{2p-2}(\langle[n]\rangle)$. Since $\cC_{2,2p}(\rho)=0$ when $\Delta$ has no induced $K_{2,2p}$, Lemma~\ref{Lem: dimension shifting} gives the following statement.

\begin{Lemma}\label{Lem: local cycles}
Let $p\geq2$ and assume that $\Delta$ has no induced subgraph isomorphic to $K_{2,2p}$. Then $\widetilde{H}_{2p-1}(\Nb(\Delta);K)$ is isomorphic to the space of families $(\sigma_v)_{v\in[n]}$, where $\sigma_v$ is a $(2p-2)$-cycle of the simplex $\langle N_\Delta(v)\rangle$ and $\sum_v\sigma_v=0$, compatibly with the inclusions induced by passing to induced subgraphs.
\end{Lemma}

\subsection{Reduction to the core and a forbidden subgraph characterization}

For an integer $k\geq1$, the $k$-\textit{core} $\core_k(\Delta)$ of a graph $\Delta$ is the largest induced subgraph of $\Delta$ of minimum degree at least $k$; it is obtained by repeatedly deleting vertices of degree less than $k$, the result does not depend on the order of the deletions, and it may be empty. We identify $\core_k(\Delta)$ with its vertex set.

\begin{Lemma}\label{Lem: cycle support}
Let $c$ be a nonzero reduced $d$-cycle of a simplicial complex, $d\geq0$. Then at least $d+2$ simplices occur in $c$ with nonzero coefficient. If $d\geq1$, then every vertex occurring in $c$ lies in at least $d+1$ of these simplices.
\end{Lemma}

\begin{proof}
Let $F$ occur in $c$. For every $x\in F$, the face $F\setminus\{x\}$ occurs in $\delta F$, so it must be cancelled by a simplex $F_x\neq F$ occurring in $c$; if $F_x=F_y$ with $x\neq y$, then $F_x\supseteq(F\setminus\{x\})\cup(F\setminus\{y\})=F$, a contradiction. For the second assertion, apply the first one to the link of the vertex in $c$, which is a nonzero reduced $(d-1)$-cycle.
\end{proof}

\begin{Proposition}\label{Prop: homology on core}
Let $p\geq2$, assume that $\Delta$ has no induced subgraph isomorphic to $K_{2,2p}$, and let $C=\core_{2p}(\Delta)$. Then the inclusion $\Nb(\Delta_C)\subseteq\Nb(\Delta)$ induces an isomorphism
\[
\widetilde{H}_{2p-1}(\Nb(\Delta_C);K)\cong\widetilde{H}_{2p-1}(\Nb(\Delta);K).
\]
\end{Proposition}

\begin{proof}
By Lemma~\ref{Lem: local cycles} it suffices to show that every family $(\sigma_v)$ as in that lemma is supported on $C$. Let $Z=\{v:\sigma_v\neq0\}$, let $Y$ be the set of vertices occurring in some $\sigma_v$ with $v\in Z$, and let $T=Z\cup Y$. If $v\in Z$, the nonzero $(2p-2)$-cycle $\sigma_v$ involves at least $2p$ vertices of $N_\Delta(v)$, all of which lie in $Y$, because on $2p-1$ vertices there is a single $(2p-2)$-simplex and its boundary is nonzero. Let $u\in Y$ and choose $v\in Z$ such that $u$ occurs in $\sigma_v$. By Lemma~\ref{Lem: cycle support}, at least $2p-1$ distinct $(2p-1)$-sets $B\ni u$ occur in $\sigma_v$. Since $\sum_w\sigma_w=0$, each such $B$ occurs in some $\sigma_{w_B}$ with $w_B\neq v$; then $w_B\in Z$ and $u\in B\subseteq N_\Delta(w_B)$. If $w_B=w_{B'}$ for $B\neq B'$, then $v$ and $w_B$ have the at least $2p$ common neighbours $B\cup B'$, and $\Delta$ contains an induced $K_{2,2p}$, a contradiction. Hence $u$ is adjacent to $v$ and to at least $2p-1$ further vertices of $T$. Thus $\Delta_T$ has minimum degree at least $2p$, and $T\subseteq C$.
\end{proof}

\begin{Proposition}\label{Prop: top homology on core}
Let $p\geq2$, assume that $\Delta$ satisfies \Bp, and let $C=\core_{2p}(\Delta)$. Then every top cycle of $\Delta^{[p]}$ is supported on $C$, that is, $\widetilde{H}_{2p-1}(\Delta_C^{[p]};K)=\widetilde{H}_{2p-1}(\Delta^{[p]};K)$.
\end{Proposition}

\begin{proof}
The inclusion $\subseteq$ is \eqref{Eq: injectivity}. Since \Bp\ passes to $\Delta_C$, Propositions~\ref{Prop: unequal class} and~\ref{Prop: homology on core}, together with Lemma~\ref{Lem: dimension shifting}(2),(3), show that $\cZ_{a,b}(\Delta_C)\to\cZ_{a,b}(\Delta)$ is an isomorphism for every type $\{a,b\}$ with $a<b$; for odd $p$ the same holds for $\cZ_{p,p}$ and, since $\tau$ commutes with these maps, for $\cZ_{p,p}^\tau$. We conclude by \eqref{Eq: splitting} and Proposition~\ref{Prop: equal class}.
\end{proof}

\begin{Corollary}\label{Coro: planar}
Let $\Delta$ be a $1$-dimensional flag simplicial complex and let $2\leq p\leq\nu(\overline{\Delta})$. If $\core_{2p}(\Delta)=\emptyset$, then $I_\Delta^{[p]}$ has a linear resolution if and only if it is linearly related. In particular, this holds for every planar $\Delta$ and every $p\geq2$.
\end{Corollary}

\begin{proof}
If $I_\Delta^{[p]}$ is linearly related, then $\Delta$ satisfies \Bp\ by Corollary~\ref{coro: lin related}, so $\widetilde{H}_{2p-1}(\Nb(\Delta);K)=0$ by Proposition~\ref{Prop: homology on core}, and Theorem~\ref{Thm: main criterion} applies. A triangle-free planar graph with $m\geq3$ vertices has at most $2m-4$ edges, hence a vertex of degree at most $3$; since induced subgraphs of planar graphs are planar, $\core_4(\Delta)=\emptyset$, and a fortiori $\core_{2p}(\Delta)=\emptyset$.
\end{proof}

The following combinatorial lemma is used to produce crowns.

\begin{Lemma}\label{Lem: triangle links}
Let $d\geq2$ and let $c$ be a nonzero $d$-cycle of a simplicial complex such that every $(d-2)$-dimensional face is contained in at most three simplices occurring in $c$. Then there exist a simplex $F$ occurring in $c$ and a vertex $w\notin F$ such that all facets of the boundary of the simplex $\langle F\cup\{w\}\rangle$ occur in $c$.
\end{Lemma}

\begin{proof}
Let $F$ occur in $c$ and, as in the proof of Lemma~\ref{Lem: cycle support}, let $F_x=(F\setminus\{x\})\cup\{w_x\}$, $x\in F$, be simplices occurring in $c$ with $w_x\notin F$. For distinct $x,y\in F$, the link of $G=F\setminus\{x,y\}$ in $c$ is a nonzero $1$-cycle with at most three edges, and it contains the distinct edges $\{x,y\}$, $\{y,w_x\}$ and $\{x,w_y\}$, coming from $F$, $F_x$ and $F_y$. A nonzero $1$-cycle with at most three edges is a triangle, since each of its vertices lies on at least two of its edges; hence $w_x=w_y$. Therefore all the vertices $w_x$ coincide with a single vertex $w$, and $F$ together with the simplices $F_x$ are the facets of the boundary of $\langle F\cup\{w\}\rangle$.
\end{proof}

\begin{Theorem}\label{Thm: core degree}
Let $\Delta$ be a $1$-dimensional flag simplicial complex, let $2\leq p\leq\nu(\overline{\Delta})$, and suppose that every vertex of $C=\core_{2p}(\Delta)$ has degree at most $2p$ in $\Delta_C$ (the case $C=\emptyset$ is allowed); for instance, suppose that $\Delta$ has maximum degree at most $2p$. Then, over every field, the following conditions are equivalent:
\begin{enumerate}
\item $I_\Delta^{[p]}$ has a linear resolution;
\item $\Delta$ has no induced subgraph isomorphic to $K_{2+i,\,2p-i}$, for an even integer $0\leq i\leq p$, or to $Cr(2p+1)$.
\end{enumerate}
\end{Theorem}

\begin{proof}
The implication (1)$\Rightarrow$(2) follows from Corollaries~\ref{coro: lin related} and~\ref{Coro: crown}. Assume (2). By Proposition~\ref{Prop: homology on core} and Theorem~\ref{Thm: main criterion}, it suffices to show that $\widetilde{H}_{2p-1}(\Nb(\Delta_C);K)=0$; note that $\Delta_C$ satisfies \Bp, contains no induced crown and has maximum degree at most $2p$. Suppose that there is a nonzero class in $\widetilde{H}_{2p-1}(\Nb(\Delta_C);K)$. Since every neighbourhood in $\Delta_C$ has at most $2p$ elements, $\dim\Nb(\Delta_C)\leq2p-1$, and this class is represented by a nonzero cycle $c$. Every simplex occurring in $c$ is a set of $2p$ vertices contained in some neighbourhood, hence equal to $N_{\Delta_C}(v)$ for a vertex $v$ of degree $2p$ in $\Delta_C$, which is unique because $\Delta$ has no induced $K_{2,2p}$. If a $(2p-2)$-set $G$ were contained in four simplices $N_{\Delta_C}(v_1),\dots,N_{\Delta_C}(v_4)$ occurring in $c$, then $(\{v_1,\dots,v_4\},G)$ would be a complete pair inducing $K_{4,\,2p-2}$, which is excluded by \Bp. By Lemma~\ref{Lem: triangle links} with $d=2p-1$, there are a set $T$ with $|T|=2p+1$ and pairwise distinct vertices $v_t$, $t\in T$, such that $N_{\Delta_C}(v_t)=T\setminus\{t\}$. Any two elements of $T$ have a common neighbour, so $T$ is independent and $v_t\notin T$; any two vertices $v_t$ have a common neighbour, so they are not adjacent. Hence $\Delta_{T\cup\{v_t:t\in T\}}\cong Cr(2p+1)$, a contradiction.
\end{proof}

Under the assumptions of Theorem~\ref{Thm: core degree}, a nonempty core is $2p$-regular, and every induced $Cr(2p+1)$, having minimum degree $2p$, is a connected component of it. Theorem~\ref{Thm: core degree} applies, for instance, to every graph obtained by attaching trees to a $2p$-regular triangle-free graph, whatever its maximum degree. Example~\ref{Exa: hypersimplex} shows that the assumption on the core cannot be dropped.

\subsection{The first nonlinear Betti numbers and monotonicity}

Without any degree assumption, the first homological degrees of the nonlinear strand are still controlled by crowns. 

\begin{Lemma}\label{Lem: small dense graphs}
Let $H$ be a triangle-free graph with minimum degree at least $2p$ and at most $4p+2$ vertices, containing no induced subgraph isomorphic to $K_{2,2p}$. Then $H$ is $2p$-regular.
\end{Lemma}

\begin{proof}
Let $v$ be a vertex of degree $d$, let $U=N_H(v)$ and let $R$ be the set of the remaining vertices different from $v$. The set $U$ is independent, so every $u\in U$ has at least $2p-1$ neighbours in $R$; in particular $|R|\geq2p-1$ and $d\leq2p+2$. Two vertices of $U$ have the common neighbour $v$, hence at most $2p-2$ common neighbours in $R$. If $|R|=2p-1$, then two vertices of $U$ have $R$ as common neighbours, which is impossible; this excludes $d=2p+2$, and also $d=2p+1$ with $|R|=2p-1$. If $d=2p+1$ and $|R|=2p$, then every $u\in U$ misses exactly one vertex of $R$, and distinct vertices of $U$ miss distinct vertices of $R$, which gives an injection of a set with $2p+1$ elements into a set with $2p$ elements. Hence $d=2p$.
\end{proof}

\begin{Theorem}\label{Thm: first Betti}
Let $\Delta$ be a $1$-dimensional flag simplicial complex and let $2\leq p\leq\nu(\overline{\Delta})$. If $I_\Delta^{[p]}$ is linearly related, then, over every field, $\beta_{j,j+2p+1}(I_\Delta^{[p]})=0$ for $0\leq j\leq2p$ and
\[
\beta_{2p+1,\,4p+2}(I_\Delta^{[p]})=p\,\bigl|\{\text{induced subgraphs of }\Delta\text{ isomorphic to }Cr(2p+1)\}\bigr|.
\]
\end{Theorem}

\begin{proof}
By Corollary~\ref{coro: lin related}, $\Delta$ satisfies \Bp, and so does every induced subgraph. Let $W\subseteq[n]$ with $|W|\leq4p+2$ and $\widetilde{H}_{2p-1}(\Delta_W^{[p]};K)\neq0$. By Theorem~\ref{Thm: dimension formula}, $\widetilde{H}_{2p-1}(\Nb(\Delta_W);K)\neq0$, and by Proposition~\ref{Prop: homology on core} the core $C_W=\core_{2p}(\Delta_W)$ satisfies $\widetilde{H}_{2p-1}(\Nb(\Delta_{C_W});K)\neq0$. By Lemma~\ref{Lem: small dense graphs}, $\Delta_{C_W}$ is $2p$-regular, and the argument in the proof of Theorem~\ref{Thm: core degree} yields an induced $Cr(2p+1)$ in $\Delta_{C_W}$. As this graph has $4p+2$ vertices, $C_W=W$ and $\Delta_W\cong Cr(2p+1)$. Hence all summands in Lemma~\ref{Lem: linear iff top homology}(1) vanish if $j\leq 2p$, that is, if $|W|\leq4p+1$, and for $|W|=4p+2$ the only nonzero summands correspond to induced crowns, each of which contributes $p$ by Corollary~\ref{Coro: crown}.
\end{proof}

In particular, if $I_\Delta^{[p]}$ is linearly related and $\Delta$ contains no induced $Cr(2p+1)$, then $\beta_{j,j+2p+1}(I_\Delta^{[p]})=0$ for $0\leq j\leq2p+1$. The following example shows that the assumption that $I_\Delta^{[p]}$ is linearly related cannot be omitted in Theorem~\ref{Thm: first Betti}.

\begin{Example}\label{exa: crown and K24}
Consider the $1$-dimensional flag simplicial complex $\Delta$ illustrated in Figure~\ref{fig: last exa}.

\begin{figure}[h]
    \centering
    \includegraphics[scale=0.6]{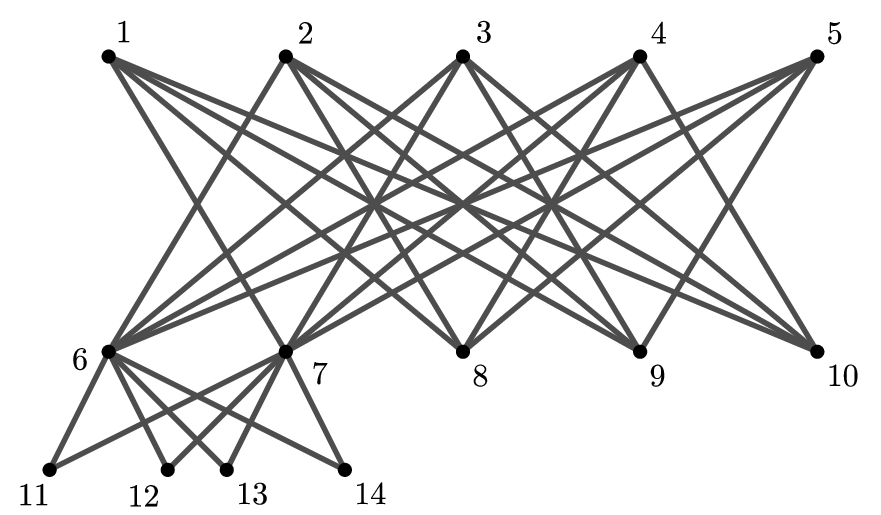}
    \caption{A $1$-dimensional flag simplicial complex containing both $Cr(5)$ and $K_{2,4}$.}
    \label{fig: last exa}
\end{figure}

The graph $\Delta$ contains exactly one induced subgraph isomorphic to $Cr(5)$, namely the one induced by $\{1,2,\dots,10\}$, and exactly $35$ induced subgraphs isomorphic to $K_{2,4}$, for instance the one induced by $\{6,7,11,12,13,14\}$. Using \textit{Macaulay2}~\cite{M2}, we obtain the following Betti table of $I_\Delta^{[2]}$ over a field of characteristic zero; in accordance with Theorem~\ref{Thm: betti 2,2p+2} we have $\beta_{1,6}=35$, but $\beta_{5,10}=1817\neq2$.
\[
\begin{array}{r|ccccccccccc}
      & 0 & 1 & 2 & 3 & 4 & 5 & 6 & 7 & 8 & 9 & 10 \\
\hline
\mathrm{total}: & 857 & 6603 & 23844 & 52459 & 77265 & 79201 & 57024 & 28409 & 9359 & 1840 & 164 \\
0: & . & . & . & . & . & . & . & . & . & . & . \\
1: & . & . & . & . & . & . & . & . & . & . & . \\
2: & . & . & . & . & . & . & . & . & . & . & . \\
3: & . & . & . & . & . & . & . & . & . & . & . \\
4: & 857 & 6568 & 23585 & 51619 & 75705 & 77384 & 55661 & 27763 & 9181 & 1818 & 164 \\
5: & . & 35 & 259 & 840 & 1560 & 1817 & 1363 & 646 & 178 & 22 & .
\end{array}
\]
\end{Example}

Finally, linear relatedness of a squarefree power forces the linearity of the next one.

\begin{Corollary}\label{Coro: lin res for p then lin res for p+1}
Let $\Delta$ be a $1$-dimensional flag simplicial complex and let $2\leq p<\nu(\overline{\Delta})$. If $I_\Delta^{[p]}$ is linearly related, then $I_\Delta^{[p+1]}$ has a linear resolution over every field. In particular, if $I_\Delta^{[p]}$ has a linear resolution, then $I_\Delta^{[q]}$ has a linear resolution for every $p\leq q\leq\nu(\overline{\Delta})$.
\end{Corollary}

\begin{proof}
By Corollary~\ref{coro: lin related}, $\Delta$ satisfies \Bp, hence also (B$_{p+1}$), and in particular it has no induced $K_{2,2p}$. By Lemma~\ref{Lem: dimension shifting}(1) applied to the adjacency relation $\rho$, $\widetilde{H}_{2p+1}(\Nb(\Delta);K)\cong\cQ_{2,2p}$, which is a quotient of a subspace of $\cC_{2,2p}(\rho)=0$. Theorem~\ref{Thm: main criterion}, applied with $p+1$ in place of $p$, gives the first assertion, and the second one follows by induction.
\end{proof}

\section{Graphs from three-dimensional complexes and the role of the characteristic}\label{Sec: characteristic}

In this section we show that, without the degree condition of Theorem~\ref{Thm: core degree}, the homological condition of Theorem~\ref{Thm: main criterion} cannot be replaced by any condition on the graph alone, because it depends on the characteristic of the base field. We work with $p=2$.

\begin{Construction}\label{Cons: Delta_M}
Let $M$ be a simplicial complex of dimension at most $3$ with vertex set $V(M)$, set of edges $E(M)$ and set of facets $\mathcal{F}(M)$. We assume that every vertex of $M$ lies in an edge, that $|E(M)|\geq4$, and that
\begin{equation}\tag{III}\label{Eq: triangle condition}
\text{every triangle of $M$ is contained in at most two facets of $M$.}
\end{equation}
We denote by $\Delta_M$ the graph with vertex set $E(M)\sqcup V(M)\sqcup\mathcal{F}(M)$ in which an edge $e\in E(M)$ is adjacent to its two endpoints and to the facets of $M$ containing it, and there are no other adjacencies. Thus $\Delta_M$ is bipartite, with bipartition $V_1\sqcup V_2$, where $V_1=E(M)$ and $V_2=V(M)\sqcup\mathcal{F}(M)$.
\end{Construction}

\begin{Example}\label{Exa: Delta_M small}
To draw $\Delta_M$, place a point at each vertex of $M$, at the midpoint of each edge and at the barycentre of each facet, and join every midpoint to the two endpoints of its edge and to the barycentres of the facets containing it; see Figure~\ref{fig: Delta_M small}.
 
\medskip
\noindent (1) Let $M$ be the simplicial complex on $[4]$ with facets $\{1,2,3\}$ and $\{1,3,4\}$, that is, two triangles glued along the edge $\{1,3\}$. Then $\Delta_M$ has $11$ vertices, namely the $4$ vertices, the $5$ edges and the $2$ facets of $M$, and $16$ edges: every edge of $M$ is joined to its two endpoints, the edge $\{1,3\}$ to both facets, and each of the other four edges to exactly one facet. The maximal faces of $\Nb_1(\Delta_M)$ are the sets of edges through the vertex $1$ or through the vertex $3$ and the sets of edges of the two triangles. They all contain $\{1,3\}$, so $\Nb_1(\Delta_M)$ is the cone with apex $\{1,3\}$ over the $4$-cycle with vertices $\{1,2\},\{1,4\},\{3,4\},\{2,3\}$; in particular, it is contractible, as $M$ is.
 
\medskip
\noindent (2) Let $M=\partial\Delta^3$ be the boundary of the tetrahedron on $[4]$, whose facets are the four triangles contained in $[4]$. Then $\Delta_M$ has $4+6+4=14$ vertices and $24$ edges, since every edge of $M$ has two endpoints and lies in two triangles. Here $\Nb_1(\Delta_M)$ is the boundary of the octahedron: its vertices are the six edges of $M$, its eight triangles are the four sets of edges through a vertex of $M$ and the four sets of edges of a triangle of $M$, and two edges of $M$ are opposite vertices of the octahedron if and only if they are disjoint. Thus $\Nb_1(\Delta_M)$ is a $2$-sphere, as $M$ is.
 
\medskip
\noindent (3) The boundary of the $4$-simplex, a triangulated $3$-sphere, is treated in Example~\ref{Exa: hypersimplex}.
 
\medskip
Proposition~\ref{Prop: Delta_M homology} below shows that, in general, $\Nb_1(\Delta_M)$ has the homology of $M$. In (1) and (2) we have $H_3(M;K)=0$, and $I_{\Delta_M}^{[2]}$ has a linear resolution over every field by Corollary~\ref{Coro: Delta_M linear}; in (3), $H_3(M;K)\cong K$ for every field $K$, and $I_{\Delta_M}^{[2]}$ never has a linear resolution.
\end{Example}
 
\begin{figure}[ht]
\centering
\subfloat[]{%
\begin{tikzpicture}[scale=1.15,vM/.style={circle,fill=black,inner sep=1.6pt},eM/.style={circle,draw=black,fill=white,inner sep=1.6pt},fM/.style={rectangle,draw=black,fill=black!35,inner sep=2.1pt}]
\coordinate (1) at (0,0); \coordinate (2) at (2.6,0); \coordinate (3) at (2.6,2.6); \coordinate (4) at (0,2.6);
\coordinate (12) at (1.3,0); \coordinate (23) at (2.6,1.3); \coordinate (13) at (1.3,1.3); \coordinate (34) at (1.3,2.6); \coordinate (14) at (0,1.3);
\coordinate (F123) at (1.733,0.867); \coordinate (F134) at (0.867,1.733);
\fill[black!8] (1)--(2)--(3)--cycle; \fill[black!8] (1)--(3)--(4)--cycle;
\draw (1)--(2)--(3)--(4)--(1)--(3);
\foreach \e in {12,23,13} \draw (\e)--(F123);
\foreach \e in {13,34,14} \draw (\e)--(F134);
\foreach \v/\p in {1/below left,2/below right,3/above right,4/above left} \node[vM,label={\p:{\small $\v$}}] at (\v) {};
\foreach \e in {12,23,13,34,14} \node[eM] at (\e) {};
\node[fM] at (F123) {}; \node[fM] at (F134) {};
\end{tikzpicture}}\qquad\qquad
\subfloat[]{%
\begin{tikzpicture}[scale=1.15,vM/.style={circle,fill=black,inner sep=1.6pt},eM/.style={circle,draw=black,fill=white,inner sep=1.6pt},fM/.style={rectangle,draw=black,fill=black!35,inner sep=2.1pt}]
\coordinate (1) at (0,0); \coordinate (2) at (3.6,0); \coordinate (3) at (1.8,3.118); \coordinate (4) at (1.8,1.039);
\coordinate (12) at (1.8,0); \coordinate (13) at (0.9,1.559); \coordinate (23) at (2.7,1.559);
\coordinate (14) at (0.9,0.520); \coordinate (24) at (2.7,0.520); \coordinate (34) at (1.8,2.078);
\coordinate (F124) at (1.8,0.346); \coordinate (F134) at (1.2,1.386); \coordinate (F234) at (2.4,1.386);
\coordinate (F123) at (1.8,-1.05);
\fill[black!8] (1)--(2)--(3)--cycle;
\draw (1)--(2)--(3)--cycle; \draw (1)--(4); \draw (2)--(4); \draw (3)--(4);
\foreach \e in {12,14,24} \draw (\e)--(F124);
\foreach \e in {13,14,34} \draw (\e)--(F134);
\foreach \e in {23,24,34} \draw (\e)--(F234);
\draw (F123)--(12);
\draw (F123) to[out=180,in=210,looseness=2.4] (13);
\draw (F123) to[out=0,in=-30,looseness=2.4] (23);
\foreach \v/\p in {1/left,2/right,3/above} \node[vM,label={\p:{\small $\v$}}] at (\v) {};
\node[vM,label={[label distance=-1pt]60:{\small $4$}}] at (4) {};
\foreach \e in {12,13,23,14,24,34} \node[eM] at (\e) {};
\foreach \f in {F124,F134,F234,F123} \node[fM] at (\f) {};
\end{tikzpicture}}
\caption{The graphs $\Delta_M$ of Example~\ref{Exa: Delta_M small}: (a) two triangles glued along an edge; (b) the boundary of the tetrahedron, drawn in the plane with $\{1,2,3\}$ as outer face, which corresponds to the square below the triangle. Black dots, white dots and grey squares correspond to the vertices, the edges and the facets of $M$, respectively; the shaded regions are triangles of $M$.}
\label{fig: Delta_M small}
\end{figure}

\begin{Lemma}\label{Lem: Delta_M B2}
The graph $\Delta_M$ is a $1$-dimensional flag simplicial complex satisfying \Btwo, and $\nu(\overline{\Delta_M})\geq2$.
\end{Lemma}

\begin{proof}
Since $\Delta_M$ is bipartite, it is triangle-free. An induced $K_{2,4}$ consists of two vertices in the same part with four common neighbours, so it suffices to show that two distinct vertices in the same part have at most three common neighbours. Two vertices of $M$ lie in at most one common edge. A vertex $v$ and a facet $F$ have as common neighbours the $|F|-1\leq3$ edges of $F$ through $v$, if $v\in F$, and no common neighbour otherwise. Two distinct facets $F,F'$ satisfy $|F\cap F'|\leq3$, because $|F|,|F'|\leq4$, so they contain at most three common edges. Finally, two distinct edges $e,e'$ have at most one common endpoint, and their common facets contain $e\cup e'$. If $|e\cup e'|=4$, the only facet containing $e\cup e'$ can be $e\cup e'$ itself; if $|e\cup e'|=3$, then $e\cup e'$ lies in at most two facets by \eqref{Eq: triangle condition}, and in this case $e$ and $e'$ have one common endpoint. In all cases $e$ and $e'$ have at most three common neighbours. Finally, $\overline{\Delta_M}$ contains the complete graph on $E(M)$, which has at least four vertices, so $\nu(\overline{\Delta_M})\geq2$.
\end{proof}

\begin{Proposition}\label{Prop: Delta_M homology}
For every $k\geq1$ there are isomorphisms
\[
\widetilde{H}_k(\Nb_1(\Delta_M);K)\cong\widetilde{H}_k(\Nb_2(\Delta_M);K)\cong H_k(M;K).
\] In particular, $\widetilde{H}_3(\Nb(\Delta_M);K)\cong H_3(M;K)^{\oplus2}$.
\end{Proposition}

\begin{proof}
By Remark~\ref{Rem: Dowker}, it suffices to prove that $\Nb_1=\Nb_1(\Delta_M)$ is homotopy equivalent to the geometric realization of $M$. The faces of $\Nb_1$ are the sets $S\subseteq E(M)$ contained in the neighbourhood of a vertex or of a facet of $M$, that is, the sets of edges which share a vertex or whose union is contained in a facet of $M$.

Let $\sd(M)$ be the barycentric subdivision of $M$, whose faces are the chains $\sigma_0\subsetneq\sigma_1\subsetneq\cdots\subsetneq\sigma_k$ of nonempty faces of $M$. For $e\in E(M)$, let $K_e=\st_{\sd(M)}(e)$ be the closed star of the vertex $e$ of $\sd(M)$; a chain belongs to $K_e$ if and only if each of its elements is comparable with $e$. We show that the subcomplexes $K_e$ cover $\sd(M)$, that their nerve is $\Nb_1$, and that all their nonempty intersections are contractible. The nerve theorem \cite[Theorem~10.6]{Bjorner1995} then gives that $\Nb_1$ is homotopy equivalent to $\sd(M)$, hence to $M$.

Let $\sigma_0\subsetneq\cdots\subsetneq\sigma_k$ be a chain. If $\dim\sigma_0\geq1$, it lies in $K_e$ for every edge $e\subseteq\sigma_0$. If $\sigma_0=\{v\}$ and $k\geq1$, it lies in $K_e$ for every edge $e$ with $v\in e\subseteq\sigma_1$. If the chain is $\{v\}$, it lies in $K_e$ for every edge $e$ containing $v$. Hence the subcomplexes $K_e$ cover $\sd(M)$.

For a nonempty set $S\subseteq E(M)$, let $\Phi_S$ be the set of nonempty faces of $M$ comparable with every $e\in S$. Then $\bigcap_{e\in S}K_e$ is the order complex of $\Phi_S$; in particular, it is nonempty if and only if $\Phi_S\neq\emptyset$. If $S=\{e\}$, then $e\in\Phi_S$ is comparable with every element of $\Phi_S$. Let $|S|\geq2$. A face comparable with two distinct edges is either a vertex contained in both or a face containing both; hence $\Phi_S$ consists of the vertices in $\bigcap_{e\in S}e$ and of the faces containing $\bigcup_{e\in S}e$. If $\bigcap_{e\in S}e=\{v\}$, then $v\in\Phi_S$ and $v$ is contained in every element of $\Phi_S$. If $\bigcap_{e\in S}e=\emptyset$, then $\Phi_S$ is the set of faces containing $\bigcup_{e\in S}e$, which is either empty or has the minimum $\bigcup_{e\in S}e$. In all cases a nonempty $\Phi_S$ has an element comparable with all the other ones, so its order complex is a cone and is contractible. Moreover, $\Phi_S\neq\emptyset$ if and only if the edges in $S$ share a vertex or $\bigcup_{e\in S}e$ is a face of $M$, that is, is contained in a facet of $M$. Thus the nerve of the family $\{K_e\}$ is $\Nb_1$.

The last assertion follows because $\Nb(\Delta_M)=\Nb_1(\Delta_M)\sqcup\Nb_2(\Delta_M)$.
\end{proof}

\begin{Corollary}\label{Coro: Delta_M linear}
Let $M$ be as in Construction~\ref{Cons: Delta_M} and let $N=|V(\Delta_M)|$. Then $I_{\Delta_M}^{[2]}$ is linearly related, and
\[
\beta_{N-5,N}(I_{\Delta_M}^{[2]})=\dim_K\widetilde{H}_3(\Delta_M^{[2]};K)=2\dim_K H_3(M;K).
\]
In particular, $I_{\Delta_M}^{[2]}$ has a linear resolution over $K$ if and only if $H_3(M;K)=0$, that is, if and only if $M$ has no nonzero $3$-cycle with coefficients in $K$.
\end{Corollary}

\begin{proof}
By Lemma~\ref{Lem: Delta_M B2} and Corollary~\ref{coro: lin related}, $I_{\Delta_M}^{[2]}$ is linearly related. By Theorem~\ref{Thm: dimension formula}(1) and Proposition~\ref{Prop: Delta_M homology}, $\dim_K\widetilde{H}_3(\Delta_M^{[2]};K)=\dim_K\widetilde{H}_3(\Nb(\Delta_M);K)=2\dim_KH_3(M;K)$, and the remaining assertions follow from Lemma~\ref{Lem: linear iff top homology}, since $\dim M\leq3$.
\end{proof}

The smallest instance already shows that induced crowns and complete bipartite graphs are not the only induced subgraphs preventing a linear resolution.

\begin{Example}\label{Exa: hypersimplex}
Let $M$ be the boundary of the $4$-simplex on $[5]$, whose facets are the sets $[5]\setminus\{i\}$, $i\in[5]$. Write $v_{jk}$ for the edge $\{j,k\}$ of $M$, $a_i$ for the vertex $i$ and $b_i$ for the facet $[5]\setminus\{i\}$. Then $\Delta_M$ is the graph on the $20$ vertices $a_i,b_i,v_{jk}$ whose $50$ edges are $\{a_i,v_{jk}\}$ for $i\in\{j,k\}$ and $\{b_i,v_{jk}\}$ for $i\notin\{j,k\}$; see Figure~\ref{fig: hypersimplex}. Equivalently, $\Delta_M$ is the vertex--facet incidence graph of the hypersimplex $\{t\in[0,1]^5:t_1+\cdots+t_5=2\}$, whose vertices are the points $e_j+e_k$ and whose facets are the five tetrahedra $\{t_i=1\}$ and the five octahedra $\{t_i=0\}$.

\begin{figure}[h]
\centering
\begin{tikzpicture}[vtx/.style={circle,fill=black,inner sep=1.4pt}]
\coordinate (a1) at (0.0000,2.7);
\coordinate (b1) at (0.0000,-2.7);
\coordinate (a2) at (2.8125,2.7);
\coordinate (b2) at (2.8125,-2.7);
\coordinate (a3) at (5.6250,2.7);
\coordinate (b3) at (5.6250,-2.7);
\coordinate (a4) at (8.4375,2.7);
\coordinate (b4) at (8.4375,-2.7);
\coordinate (a5) at (11.2500,2.7);
\coordinate (b5) at (11.2500,-2.7);
\coordinate (v12) at (0.0000,0);
\coordinate (v13) at (1.2500,0);
\coordinate (v14) at (2.5000,0);
\coordinate (v15) at (3.7500,0);
\coordinate (v23) at (5.0000,0);
\coordinate (v24) at (6.2500,0);
\coordinate (v25) at (7.5000,0);
\coordinate (v34) at (8.7500,0);
\coordinate (v35) at (10.0000,0);
\coordinate (v45) at (11.2500,0);
\draw[thick] (a1) -- (v12);
\draw[thick] (a1) -- (v13);
\draw[thick] (a1) -- (v14);
\draw[thick] (a1) -- (v15);
\draw[thick] (a2) -- (v12);
\draw[thick] (a2) -- (v23);
\draw[thick] (a2) -- (v24);
\draw[thick] (a2) -- (v25);
\draw[thick] (a3) -- (v13);
\draw[thick] (a3) -- (v23);
\draw[thick] (a3) -- (v34);
\draw[thick] (a3) -- (v35);
\draw[thick] (a4) -- (v14);
\draw[thick] (a4) -- (v24);
\draw[thick] (a4) -- (v34);
\draw[thick] (a4) -- (v45);
\draw[thick] (a5) -- (v15);
\draw[thick] (a5) -- (v25);
\draw[thick] (a5) -- (v35);
\draw[thick] (a5) -- (v45);
\draw[gray!70] (b1) -- (v23);
\draw[gray!70] (b1) -- (v24);
\draw[gray!70] (b1) -- (v25);
\draw[gray!70] (b1) -- (v34);
\draw[gray!70] (b1) -- (v35);
\draw[gray!70] (b1) -- (v45);
\draw[gray!70] (b2) -- (v13);
\draw[gray!70] (b2) -- (v14);
\draw[gray!70] (b2) -- (v15);
\draw[gray!70] (b2) -- (v34);
\draw[gray!70] (b2) -- (v35);
\draw[gray!70] (b2) -- (v45);
\draw[gray!70] (b3) -- (v12);
\draw[gray!70] (b3) -- (v14);
\draw[gray!70] (b3) -- (v15);
\draw[gray!70] (b3) -- (v24);
\draw[gray!70] (b3) -- (v25);
\draw[gray!70] (b3) -- (v45);
\draw[gray!70] (b4) -- (v12);
\draw[gray!70] (b4) -- (v13);
\draw[gray!70] (b4) -- (v15);
\draw[gray!70] (b4) -- (v23);
\draw[gray!70] (b4) -- (v25);
\draw[gray!70] (b4) -- (v35);
\draw[gray!70] (b5) -- (v12);
\draw[gray!70] (b5) -- (v13);
\draw[gray!70] (b5) -- (v14);
\draw[gray!70] (b5) -- (v23);
\draw[gray!70] (b5) -- (v24);
\draw[gray!70] (b5) -- (v34);
\node[vtx,label=above:{$a_1$}] at (a1) {};
\node[vtx,label=below:{$b_1$}] at (b1) {};
\node[vtx,label=above:{$a_2$}] at (a2) {};
\node[vtx,label=below:{$b_2$}] at (b2) {};
\node[vtx,label=above:{$a_3$}] at (a3) {};
\node[vtx,label=below:{$b_3$}] at (b3) {};
\node[vtx,label=above:{$a_4$}] at (a4) {};
\node[vtx,label=below:{$b_4$}] at (b4) {};
\node[vtx,label=above:{$a_5$}] at (a5) {};
\node[vtx,label=below:{$b_5$}] at (b5) {};
\node[vtx,label={[fill=white,inner sep=0.6pt]right:{\scriptsize $v_{12}$}}] at (v12) {};
\node[vtx,label={[fill=white,inner sep=0.6pt]right:{\scriptsize $v_{13}$}}] at (v13) {};
\node[vtx,label={[fill=white,inner sep=0.6pt]right:{\scriptsize $v_{14}$}}] at (v14) {};
\node[vtx,label={[fill=white,inner sep=0.6pt]right:{\scriptsize $v_{15}$}}] at (v15) {};
\node[vtx,label={[fill=white,inner sep=0.6pt]right:{\scriptsize $v_{23}$}}] at (v23) {};
\node[vtx,label={[fill=white,inner sep=0.6pt]right:{\scriptsize $v_{24}$}}] at (v24) {};
\node[vtx,label={[fill=white,inner sep=0.6pt]right:{\scriptsize $v_{25}$}}] at (v25) {};
\node[vtx,label={[fill=white,inner sep=0.6pt]right:{\scriptsize $v_{34}$}}] at (v34) {};
\node[vtx,label={[fill=white,inner sep=0.6pt]right:{\scriptsize $v_{35}$}}] at (v35) {};
\node[vtx,label={[fill=white,inner sep=0.6pt]right:{\scriptsize $v_{45}$}}] at (v45) {};
\end{tikzpicture}
\caption{The graph $\Delta_M$ for the boundary $M$ of the $4$-simplex.}
\label{fig: hypersimplex}
\end{figure}

By Lemma~\ref{Lem: Delta_M B2}, $\Delta_M$ has no induced $K_{2,4}$. It has no induced $Cr(5)$ either. Indeed, since $Cr(5)$ is connected and bipartite, an induced crown would consist of five vertices $v_{jk}$ and of a set $T$ of five vertices among the $a_i$ and the $b_i$, such that every chosen $v_{jk}$ is adjacent to exactly four vertices of $T$ and any two vertices of $T$ have exactly three common neighbours among the chosen $v_{jk}$. Two vertices $a_i,a_{i'}$ have only one common neighbour, so $T$ contains at most one vertex $a_t$. Each $v_{jk}$ is adjacent to exactly three vertices $b_i$, so each chosen $v_{jk}$ must be adjacent to $a_t$; but then $a_t$ has five neighbours in the crown, which is impossible.

On the other hand, $M$ is a triangulated $3$-sphere, so $H_3(M;K)\cong K$ for every field $K$. By Corollary~\ref{Coro: Delta_M linear}, $I_{\Delta_M}^{[2]}$ does not have a linear resolution over any field, and $\beta_{15,20}(I_{\Delta_M}^{[2]})=2$. Note that $\Delta_M$ has minimum degree $4$, so it coincides with its $4$-core, and its maximum degree is $6$; thus Example~\ref{Exa: hypersimplex} is consistent with Theorem~\ref{Thm: core degree}, and shows that the assumption on the core in that theorem cannot be removed. By Theorem~\ref{Thm: minimal forbidden} below, every proper induced subgraph of $\Delta_M$ has a linear second squarefree power, so $\Delta_M$ is a minimal forbidden graph in the sense of Section~\ref{Subsec: minimal forbidden}.
\end{Example}

We now come to the dependence on the characteristic. Let $P_6$ be the six-vertex triangulation of the real projective plane, shown in Figure~\ref{fig: P6}, with triangles
\[
123,\ 134,\ 145,\ 156,\ 126,\ 235,\ 245,\ 246,\ 346,\ 356,
\]
and let $M_0=P_6*\langle\{7\},\{8\}\rangle$ be its suspension, whose facets are the twenty tetrahedra $t\cup\{7\}$ and $t\cup\{8\}$, where $t$ runs over the triangles of $P_6$. Since every edge of $P_6$ lies in exactly two triangles, every triangle of $M_0$ lies in exactly two tetrahedra, so $M_0$ satisfies \eqref{Eq: triangle condition}. Moreover, $M_0$ has $8$ vertices and $27$ edges, and by the suspension isomorphism (see, e.g., \cite{hatcher})
\[
H_3(M_0;K)\cong H_2(P_6;K)\cong\begin{cases}K&\text{if }\chara K=2,\\0&\text{otherwise.}\end{cases}
\]
\begin{figure}[ht]
\centering
\begin{tikzpicture}[scale=1.3,vM/.style={circle,fill=black,inner sep=1.6pt}]
\foreach \a/\n in {90/a,150/b,210/c,270/d,330/e,30/f} \coordinate (\n) at (\a:2);
\foreach \a/\n in {90/i1,210/i2,330/i3} \coordinate (\n) at (\a:0.8);
\fill[black!8] (a)--(b)--(c)--(d)--(e)--(f)--cycle;
\draw[thick] (a)--(b)--(c)--(d)--(e)--(f)--cycle;
\draw (i1)--(i2)--(i3)--cycle;
\foreach \o in {a,b,f} \draw (i1)--(\o);
\foreach \o in {b,c,d} \draw (i2)--(\o);
\foreach \o in {d,e,f} \draw (i3)--(\o);
\foreach \n/\l/\p in {a/5/above,b/6/above left,c/4/below left,d/5/below,e/6/below right,f/4/above right} \node[vM,label={\p:{\small $\l$}}] at (\n) {};
\node[vM,label={[label distance=-2pt]100:{\small $1$}}] at (i1) {};
\node[vM,label={[label distance=-2pt]200:{\small $2$}}] at (i2) {};
\node[vM,label={[label distance=-2pt]340:{\small $3$}}] at (i3) {};
\end{tikzpicture}
\caption{The six-vertex triangulation $P_6$ of the real projective plane; opposite points on the boundary of the hexagon are identified.}
\label{fig: P6}
\end{figure}
\begin{Theorem}\label{Thm: characteristic}
The graph $\Delta_0=\Delta_{M_0}$ is a $1$-dimensional flag simplicial complex with $55$ vertices and $174$ edges, and the following hold.
\begin{enumerate}
\item If $\chara K\neq2$, then $I_{\Delta_0}^{[2]}$ has a linear resolution over $K$.
\item If $\chara K=2$, then $I_{\Delta_0}^{[2]}$ does not have a linear resolution over $K$, and $\beta_{50,55}(I_{\Delta_0}^{[2]})=2$.
\item Over every field, $I_{\Delta_0}^{[2]}$ is linearly related and $I_{\Delta_0}^{[q]}$ has a linear resolution for every $3\leq q\leq\nu(\overline{\Delta_0})$.
\end{enumerate}
\end{Theorem}

\begin{proof}
The graph $\Delta_0$ has $27+8+20=55$ vertices, and each edge $e$ of $M_0$ is adjacent in $\Delta_0$ to its two endpoints and to the tetrahedra containing it, which gives $2\cdot27+6\cdot20=174$ edges. Assertions (1) and (2) follow from Corollary~\ref{Coro: Delta_M linear} and the computation of $H_3(M_0;K)$ above. Assertion (3) follows from Lemma~\ref{Lem: Delta_M B2}, Corollary~\ref{coro: lin related} and Corollary~\ref{Coro: lin res for p then lin res for p+1}.
\end{proof}

By Theorem~\ref{Thm: minimal forbidden} below, over $\mathbb{F}_2$ every proper induced subgraph of $\Delta_0$ has a linear second squarefree power; hence, in characteristic two, $\beta_{50,55}$ is the only nonzero entry in the row of degree $5$ of the Betti table of $I_{\Delta_0}^{[2]}$.

Theorem~\ref{Thm: characteristic} shows that, unlike in Fr\"oberg's theorem and in Theorem~\ref{Thm:linrelated1}, no combinatorial condition on $\Delta$ alone can characterize linearity.

\begin{Corollary}\label{Coro: no characteristic free}
There is no class $\mathcal{P}$ of finite graphs such that, for every field $K$ and every $1$-dimensional flag simplicial complex $\Delta$ with $\nu(\overline{\Delta})\geq2$, the ideal $I_\Delta^{[2]}$ has a linear resolution over $K$ if and only if $\Delta\in\mathcal{P}$.
\end{Corollary}

\begin{proof}
Such a class would contain $\Delta_0$, by Theorem~\ref{Thm: characteristic}(1) with $K=\mathbb{Q}$, and would not contain it, by Theorem~\ref{Thm: characteristic}(2) with $K=\mathbb{F}_2$.
\end{proof}

Nevertheless, the dependence on the characteristic is rather restricted.

\begin{Proposition}\label{Prop: exceptional characteristics}
Let $\Delta$ be a $1$-dimensional flag simplicial complex and let $2\leq p\leq\nu(\overline{\Delta})$.
\begin{enumerate}
\item If $I_\Delta^{[p]}$ has a linear resolution over some field, then it has a linear resolution over every field of characteristic zero.
\item Assume that $I_\Delta^{[p]}$ has a linear resolution over $\mathbb{Q}$. Then $I_\Delta^{[p]}$ has a linear resolution over a field of characteristic $q>0$ if and only if $q$ does not divide the order of the torsion subgroup of $H_{2p-2}(\Delta^{[p]};\mathbb{Z})$. In particular, $I_\Delta^{[p]}$ has a linear resolution over all fields except those of finitely many characteristics.
\end{enumerate}
\end{Proposition}

\begin{proof}
Since $\dim\Delta^{[p]}\leq2p-1$, the group $H_{2p-1}(\Delta^{[p]};\mathbb{Z})$ is free, and by the universal coefficient theorem \cite[Theorem~3A.3]{hatcher}
\[
\widetilde{H}_{2p-1}(\Delta^{[p]};K)\cong\bigl(H_{2p-1}(\Delta^{[p]};\mathbb{Z})\otimes K\bigr)\oplus\Tor_1^{\mathbb{Z}}\bigl(H_{2p-2}(\Delta^{[p]};\mathbb{Z}),K\bigr).
\]
The first summand vanishes for some field if and only if $H_{2p-1}(\Delta^{[p]};\mathbb{Z})=0$, that is, if and only if it vanishes for $K=\mathbb{Q}$; the second summand vanishes if and only if $\chara K$ does not divide the order of the torsion subgroup of $H_{2p-2}(\Delta^{[p]};\mathbb{Z})$. We conclude by Lemma~\ref{Lem: linear iff top homology}(2).
\end{proof}

\begin{Remark}\label{Rem: only 2}
For the graphs $\Delta_M$ of Construction~\ref{Cons: Delta_M}, only the characteristic $2$ can be exceptional. Indeed, by Corollary~\ref{Coro: Delta_M linear}, $I_{\Delta_M}^{[2]}$ has a linear resolution over $K$ if and only if $\delta_3\colon C_3(M)\to C_2(M)$ is injective, and by \eqref{Eq: triangle condition} every row of the integer matrix of $\delta_3$ has at most two nonzero entries, each equal to $\pm1$. Every square submatrix of such a matrix has determinant $0$ or $\pm2^r$ for some $r\geq0$: expanding along rows with at most one nonzero entry, one reduces to the case in which every row has exactly two nonzero entries, where the submatrix is the incidence matrix of a signed graph with as many edges as vertices, whose determinant vanishes unless every connected component contains exactly one cycle, each component then contributing a factor $0$ or $\pm2$. Hence the rank of $\delta_3$ over $K$ equals its rank over $\mathbb{Q}$ whenever $\chara K\neq2$.
\end{Remark}

\subsection{Minimal forbidden graphs}\label{Subsec: minimal forbidden}

We say that a $1$-dimensional flag simplicial complex $\Delta$ is a \textit{minimal forbidden graph} over $K$ if $I_\Delta^{[2]}$ does not have a linear resolution over $K$, while $I_{\Delta_W}^{[2]}$ has a linear resolution over $K$ or is zero for every $W\subsetneq V(\Delta)$. By \eqref{Eq: injectivity} and Lemma~\ref{Lem: linear iff top homology}(2), if $I_\Delta^{[2]}$ has a linear resolution or is zero, then the same holds for every induced subgraph of $\Delta$; hence $I_\Delta^{[2]}$ has a linear resolution over $K$ or is zero if and only if $\Delta$ has no induced subgraph isomorphic to a minimal forbidden graph over $K$. By Corollary~\ref{coro: lin related} and Theorem~\ref{Thm: main criterion}, they are $K_{2,4}$ and the minimal graphs satisfying \Btwo\ whose neighbourhood complex has nonzero third homology; $Cr(5)$ is one of them over every field, because its proper induced subgraphs have empty $4$-core and Corollary~\ref{Coro: planar} applies to them. Recall that a \textit{closed $3$-pseudomanifold} is a pure $3$-dimensional simplicial complex in which every triangle lies in exactly two facets; it is \textit{strongly connected} if any two facets can be joined by a sequence of facets in which any two consecutive ones share a triangle.

\begin{Theorem}\label{Thm: minimal forbidden}
Let $M$ be a strongly connected closed $3$-pseudomanifold such that $H_3(M;K)\neq0$. Then $\Delta_M$ is a minimal forbidden graph over $K$. In particular, for every triangulated $3$-sphere $M$, the graph $\Delta_M$ is a minimal forbidden graph over every field.
\end{Theorem}

\begin{proof}
The complex $M$ satisfies the assumptions of Construction~\ref{Cons: Delta_M}, and $I_{\Delta_M}^{[2]}$ does not have a linear resolution over $K$ by Corollary~\ref{Coro: Delta_M linear}. For a vertex $x$ of $\Delta_M$, let $\Delta_M\setminus x$ be the induced subgraph of $\Delta_M$ on $V(\Delta_M)\setminus\{x\}$. Every proper induced subgraph of $\Delta_M$ is an induced subgraph of some $\Delta_M\setminus x$, so by \eqref{Eq: injectivity} and Lemma~\ref{Lem: linear iff top homology}(2) it suffices to show that $I_{\Delta_M\setminus x}^{[2]}$ has a linear resolution over $K$ for every $x$. The graph $\Delta_M\setminus x$ is bipartite, it satisfies \Btwo\ by Lemma~\ref{Lem: Delta_M B2}, and its complement contains a complete graph on at least five elements of $E(M)$, so $\nu(\overline{\Delta_M\setminus x})\geq2$. By Theorem~\ref{Thm: main criterion} and Remark~\ref{Rem: Dowker}, it is therefore enough to prove that $\widetilde{H}_3(\Nb_1(\Delta_M\setminus x);K)=0$, where $\Nb_1$ denotes the component of the neighbourhood complex whose vertices are the elements of $E(M)$ different from $x$.

We use the notation of the proof of Proposition~\ref{Prop: Delta_M homology}, and we denote by $X$ the subcomplex of $\sd(M)$ formed by the chains not containing $x$. We first show that $\Nb_1(\Delta_M\setminus x)$ is homotopy equivalent to $X$.

Assume that $x\in E(M)$. Two distinct edges of $M$ are not comparable, so $K_e\subseteq X$ for every $e\in E(M)\setminus\{x\}$. Conversely, let $\sigma_0\subsetneq\cdots\subsetneq\sigma_k$ be a chain in $X$. If $\dim\sigma_0\geq1$, then $\sigma_0\neq x$ contains an edge different from $x$. If $\sigma_0=\{v\}$ and $k\geq1$, then either $\sigma_1\neq x$ is an edge, or $\sigma_1$ contains two edges through $v$. If the chain is $\{v\}$, then $v$ lies in at least three edges, because $M$ is pure of dimension $3$. As in the proof of Proposition~\ref{Prop: Delta_M homology}, it follows that the subcomplexes $K_e$ with $e\neq x$ cover $X$. Their nerve is the induced subcomplex of $\Nb_1(\Delta_M)$ on $E(M)\setminus\{x\}$, which is $\Nb_1(\Delta_M\setminus x)$, and their nonempty intersections are contractible by the proof of Proposition~\ref{Prop: Delta_M homology}.

Assume that $x\in V(M)\sqcup\mathcal{F}(M)$, and put $K'_e=K_e\cap X$ for $e\in E(M)$. These subcomplexes cover $X$, because the $K_e$ cover $\sd(M)$, and for every nonempty $S\subseteq E(M)$ the intersection $\bigcap_{e\in S}K'_e$ is the order complex of $\Phi_S\setminus\{x\}$. If $S=\{e\}$, then $e\in\Phi_S\setminus\{x\}$ is comparable with all the other elements, and $e$ has an endpoint different from $x$. Let $|S|\geq2$ and $\tau=\bigcup_{e\in S}e$. Recall that $\Phi_S$ consists of the vertices in $\bigcap_{e\in S}e$, of which there is at most one, and of the faces of $M$ containing $\tau$. If $\tau$ is a face of $M$ different from $x$, then $\tau\in\Phi_S\setminus\{x\}$ is comparable with every element of $\Phi_S$, and some facet of $M$ different from $x$ contains $\tau$: either $\tau$ is itself a facet, or $\tau$ is a triangle, which lies in two facets. Otherwise, $\Phi_S\setminus\{x\}$ consists of the vertex $v$ with $\bigcap_{e\in S}e=\{v\}$ if such a vertex exists and $v\neq x$, and it is empty in all other cases; indeed, if $\tau=x$, then $x$ is the only face of $M$ containing $\tau$. Hence $\Phi_S\setminus\{x\}$ is either empty or has an element comparable with all the other ones, and it is nonempty if and only if the edges in $S$ have a common neighbour in $\Delta_M\setminus x$. Thus the nerve of the family $\{K'_e\}$ is $\Nb_1(\Delta_M\setminus x)$, and all its nonempty intersections are contractible.

In both cases, the nerve theorem \cite[Theorem~10.6]{Bjorner1995} gives $\Nb_1(\Delta_M\setminus x)\simeq X$, and it remains to show that $H_3(X;K)=0$. The barycentric subdivision $\sd(M)$ is again a strongly connected closed $3$-pseudomanifold. Let $z$ be a $3$-cycle of $X$ with coefficients in $K$; it is also a $3$-cycle of $\sd(M)$. If two $3$-simplices of $\sd(M)$ share a triangle, then no other $3$-simplex contains this triangle, so the vanishing of its coefficient in $\delta_3(z)$ shows that the coefficients of $z$ on the two $3$-simplices agree up to sign. By strong connectedness, either all the coefficients of $z$ are nonzero or $z=0$. Since $M$ is pure, the vertex $x$ of $\sd(M)$ lies in some $3$-simplex of $\sd(M)$, which does not belong to $X$; hence $z=0$. As $X$ has no $4$-simplices, $H_3(X;K)=0$.

Finally, if $M$ is a triangulated $3$-sphere, then $H_3(M;K)\cong K$ for every field $K$, and the last assertion follows.
\end{proof}

The graph of Example~\ref{Exa: hypersimplex} is the case $M=\partial\Delta^4$ of Theorem~\ref{Thm: minimal forbidden}. Applied to $M_0$, Theorem~\ref{Thm: minimal forbidden} shows that $\Delta_0$ is a minimal forbidden graph over $\mathbb{F}_2$, although $I_{\Delta_0}^{[2]}$ has a linear resolution over $\mathbb{Q}$. More generally, it applies over $\mathbb{F}_2$ to every triangulation of a connected closed non-orientable $3$-manifold.

\begin{Corollary}\label{Coro: infinitely many}
For every field $K$ there are infinitely many minimal forbidden graphs over $K$. Consequently, there is no finite set $\mathcal{L}$ of graphs such that, for every $1$-dimensional flag simplicial complex $\Delta$, the ideal $I_\Delta^{[2]}$ has a linear resolution over $K$ or is zero if and only if $\Delta$ has no induced subgraph isomorphic to a graph in $\mathcal{L}$.
\end{Corollary}

\begin{proof}
The boundary of the cyclic $4$-polytope with $m\geq5$ vertices is a triangulated $3$-sphere with $m$ vertices, so Theorem~\ref{Thm: minimal forbidden} gives minimal forbidden graphs with arbitrarily many vertices. If a finite set $\mathcal{L}$ as in the statement existed, every minimal forbidden graph $\Delta$ would contain an induced subgraph isomorphic to a graph in $\mathcal{L}$, and this subgraph cannot be proper; hence $\Delta$ would be isomorphic to a graph in $\mathcal{L}$, which is impossible.
\end{proof}

\section{The shape of the Betti table and open questions}\label{Sec: final}

In this final section we describe the possible shapes of the Betti table of $I_\Delta^{[p]}$ and collect some open problems.

\begin{Proposition}\label{Prop: non van betti}
Let $\Delta$ be a $1$-dimensional flag simplicial complex on $[n]$, and let $2 \leq p \leq \nu(\overline{\Delta})$.
\begin{enumerate}
\item If $\Delta$ contains an induced subgraph isomorphic to $K_{2+i,\,2p-i}$ for some even integer $0 \leq i \leq p$, then $\beta_{j,j+2p+1}(I_\Delta^{[p]}) \neq 0$ for all $1 \leq j \leq n-2p-1$.
\item If $\Delta$ satisfies \Bp\ and contains an induced subgraph isomorphic to $Cr(2p+1)$, then $\beta_{j,j+2p+1}(I_\Delta^{[p]}) = 0$ for all $0 \leq j \leq 2p$ and $\beta_{j,j+2p+1}(I_\Delta^{[p]}) \neq 0$ for all $2p+1 \leq j \leq n-2p-1$.
\item If $\Delta$ satisfies \Bp\ and contains no induced subgraph isomorphic to $Cr(2p+1)$, then $\beta_{j,j+2p+1}(I_\Delta^{[p]})=0$ for all $0\leq j\leq 2p+1$. Moreover, either $I_\Delta^{[p]}$ has a linear resolution, or there is an integer $2p+2\leq j_0\leq n-2p-1$ such that $\beta_{j,j+2p+1}(I_\Delta^{[p]})\neq0$ if and only if $j_0\leq j\leq n-2p-1$.
\end{enumerate}
\end{Proposition}

\begin{proof}
(1) and the nonvanishing part of (2) follow from Lemma~\ref{Lem: linear iff top homology}(3), since $\beta_{1,2p+2}(I_\Delta^{[p]})\neq0$ in case (1) by Theorem~\ref{Thm: betti 2,2p+2}, and $\beta_{2p+1,4p+2}(I_\Delta^{[p]})\neq0$ in case (2) by Theorem~\ref{Thm: first Betti}. The vanishing statements in (2) and (3) follow from Theorem~\ref{Thm: first Betti}. Finally, in case (3), if $I_\Delta^{[p]}$ does not have a linear resolution, then $\beta_{n-2p-1,n}(I_\Delta^{[p]})=\dim_K\widetilde{H}_{2p-1}(\Delta^{[p]};K)\neq0$ by Lemma~\ref{Lem: linear iff top homology}, and the existence of $j_0$ follows from Lemma~\ref{Lem: linear iff top homology}(3).
\end{proof}

We summarize what our results say about the minimal graded free resolution of $I_\Delta^{[p]}$.

\begin{Discussion}\label{Disc. summarize}
Let $\Delta$ be a $1$-dimensional flag simplicial complex on $[n]$, let $2 \leq p \leq \nu(\overline{\Delta})$, and assume that $\overline{\Delta}$ has more than one $p$-matching, since otherwise $I_{\Delta}^{[p]}$ is principal. By Lemma~\ref{Lem: linear iff top homology}(1), only the rows of degrees $2p$ and $2p+1$ of the Betti table of $I_\Delta^{[p]}$ can be nonzero, and by Theorem~\ref{Thm: pd} the resolution has length $n-2p-1$ if $\Delta\cong K_{1,n-1}$ and $n-2p$ otherwise. By Corollary~\ref{coro: lin related} and Theorem~\ref{Thm: main criterion}, the row of degree $2p+1$ vanishes if and only if $\Delta$ satisfies \Bp\ and $\widetilde{H}_{2p-1}(\Nb(\Delta);K)=0$; when every vertex of $\core_{2p}(\Delta)$ has degree at most $2p$ in the core, this happens if and only if $\Delta$ satisfies \Bp\ and has no induced $Cr(2p+1)$, by Theorem~\ref{Thm: core degree}. Otherwise, by Proposition~\ref{Prop: non van betti}, exactly one of the following occurs.
\begin{enumerate}
\item[(a)] $\Delta$ has an induced $K_{2+i,\,2p-i}$ for some even $0\leq i\leq p$. The Betti table is as in Figure~\ref{Fig:BettiShapes1}, and $\beta_{1,2p+2}(I_{\Delta}^{[p]})$ is the number of such induced subgraphs, by Theorem~\ref{Thm: betti 2,2p+2}.
\begin{figure}[h]
\centering
$\begin{array}{r|cccccc}
      & 0 & 1 & \cdots & n-2p-1 & n-2p \\
\hline
\text{total:} & \beta_{0} & \beta_{1} & \cdots & \beta_{n-2p-1} & \beta_{n-2p} \\
\hline
2p:   & \beta_{0, 2p} & \beta_{1, 2p+1} & \cdots & \beta_{n-2p-1, n-1} & \beta_{n-2p, n} \\
2p+1: & 0 & \beta_{1, 2p+2} & \cdots & \beta_{n-2p-1, n} & 0 \\
\end{array}$
\caption{Betti table of $I_{\Delta}^{[p]}$ in the presence of an induced $K_{2+i, 2p-i}$.}
\label{Fig:BettiShapes1}
\end{figure}
\item[(b)] $\Delta$ satisfies \Bp\ and has an induced $Cr(2p+1)$. The Betti table is as in Figure~\ref{Fig:BettiShapes2}, and $\beta_{2p+1,4p+2}(I_{\Delta}^{[p]})$ is $p$ times the number of induced crowns, by Theorem~\ref{Thm: first Betti}.
\begin{figure}[h]
\centering
$\begin{array}{r|ccccccccc}
      & 0 & 1 & \cdots & 2p & 2p+1 & \cdots & n-2p-1 & n-2p \\
\hline
\text{total:} & \beta_{0} & \beta_{1} & \cdots & \beta_{2p} & \beta_{2p+1} & \cdots & \beta_{n-2p-1} & \beta_{n-2p} \\
\hline
2p:   & \beta_{0, 2p} & \beta_{1, 2p+1} & \cdots & \beta_{2p, 4p} & \beta_{2p+1, 4p+1} & \cdots & \beta_{n-2p-1, n-1} & \beta_{n-2p, n} \\
2p+1: & 0 & 0 & \cdots & 0 & \beta_{2p+1, 4p+2} & \cdots & \beta_{n-2p-1, n} & 0 \\
\end{array}$
\caption{Betti table of $I_{\Delta}^{[p]}$ in the presence of an induced $Cr(2p+1)$.}
\label{Fig:BettiShapes2}
\end{figure}
\item[(c)] $\Delta$ satisfies \Bp, has no induced $Cr(2p+1)$ and $\widetilde{H}_{2p-1}(\Nb(\Delta);K)\neq0$. The nonzero entries of the row of degree $2p+1$ are $\beta_{j,j+2p+1}(I_\Delta^{[p]})$ with $j_0\leq j\leq n-2p-1$ for some $j_0\geq2p+2$. By Theorem~\ref{Thm: core degree}, some vertex of $\core_{2p}(\Delta)$ has degree larger than $2p$ in the core. This case does occur: $j_0=15$ for the graph of Example~\ref{Exa: hypersimplex}, and $j_0=50$ for the graph $\Delta_0$ of Theorem~\ref{Thm: characteristic} in characteristic two.
\end{enumerate}
\end{Discussion}

\subsection{Some open questions}\label{SubSec:open question}

In the language of edge ideals, our results concern graphs $G=\overline{\Delta}$ whose complement is triangle-free. The case in which $\overline{G}$ contains triangles is much more difficult to handle. For example, computations with \textit{Macaulay2} show that, if $|V(G)|\leq6$ and $\overline{G}$ contains triangles, then $I(G)^{[2]}$ has a linear resolution over a field of characteristic zero when $\overline{G}$ is one of the graphs shown in Figure~\ref{fig:triangle_graphs}.

\begin{figure}[h]
\centering
\subfloat{\includegraphics[scale=0.4]{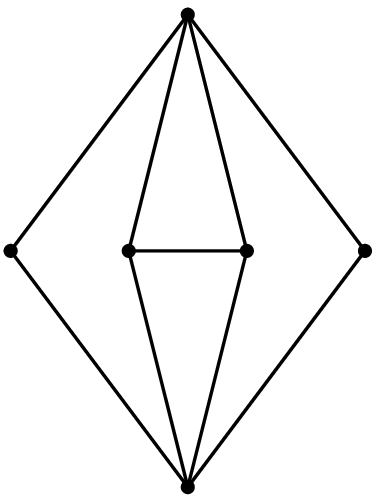}}\quad
\subfloat{\includegraphics[scale=0.4]{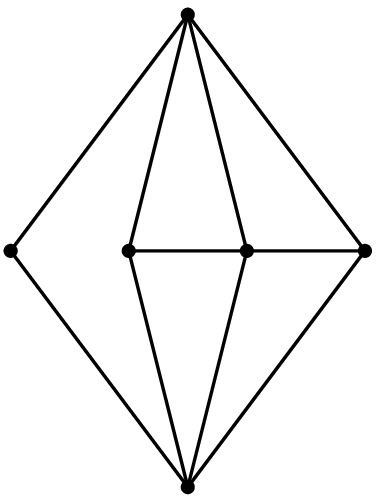}}\quad
\subfloat{\includegraphics[scale=0.4]{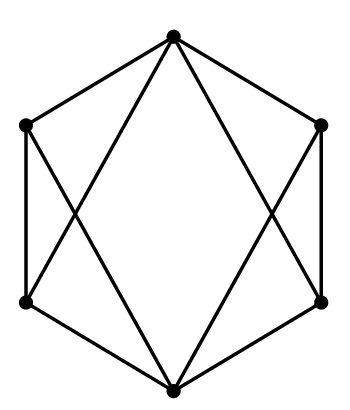}}\quad
\subfloat{\includegraphics[scale=0.4]{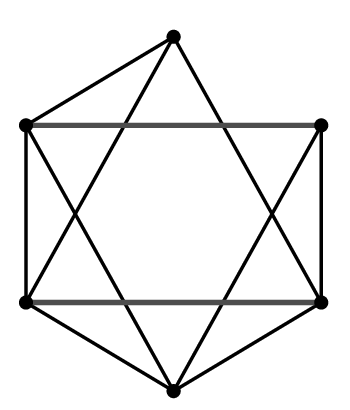}}\quad
\subfloat{\includegraphics[scale=0.4]{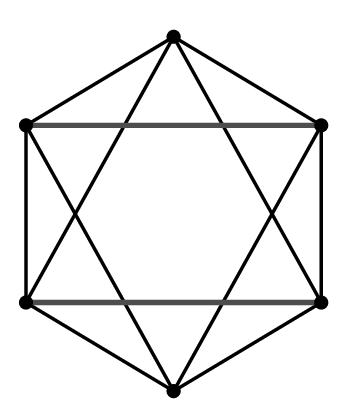}}
\caption{The complements $\overline{G}$, containing triangles, of the graphs $G$ with $|V(G)|\leq 6$ such that $I(G)^{[2]}$ has a linear resolution in characteristic zero.}
\label{fig:triangle_graphs}
\end{figure}

\begin{Question}\label{Q: triangles}
Is it possible to give a combinatorial characterization of the graphs $G$ such that $\overline{G}$ contains triangles and $I(G)^{[p]}$ has a linear resolution, for $2 \leq p \leq \nu(G)$? Can such graphs be described through the topology of a suitable complex attached to $G$, in the spirit of Theorem~\ref{Thm: main criterion}?
\end{Question}

In Example~\ref{exa1}, the complement of $G_1$ consists of an isolated vertex and the leftmost graph in Figure~\ref{fig:triangle_graphs}. Since $\beta_{1,6}(I(G_1)^{[2]})=2$, while $G_1$ has only one induced subgraph as in Theorem~\ref{Thm:linrelated1}, an analogue of Theorem~\ref{Thm: betti 2,2p+2} does not hold for arbitrary graphs.

\begin{Question}\label{Q: Betti triangles}
Can one give a combinatorial description of the graded Betti numbers in the lowest homological degree in which $I(G)^{[p]}$ fails to be linearly related, or to have a linear resolution, for all graphs $G$ such that $\overline{G}$ contains triangles?
\end{Question}

By Corollary~\ref{Coro: lin res for p then lin res for p+1}, if $\overline{G}$ is triangle-free, then the linearity of $I(G)^{[p]}$ implies that of all higher squarefree powers. Moreover, computations with \textit{Macaulay2} show that, for all graphs $G$ with $|V(G)|\leq 8$ whose complements contain triangles (approximately $10{,}000$ graphs), if $I(G)^{[2]}$ has a linear resolution over a field of characteristic zero, then every higher squarefree power also has a linear resolution.

\begin{Question}\label{Q: monotonicity}
Let $G$ be a graph such that $\overline{G}$ contains triangles. If $I(G)^{[p]}$ has a linear resolution for some $2 \leq p \leq \nu(G)$, does $I(G)^{[q]}$ have a linear resolution for all $p \leq q \leq \nu(G)$?
\end{Question}

Theorem~\ref{Thm: characteristic} and Proposition~\ref{Prop: exceptional characteristics} raise the following questions for graphs with triangle-free complement.

\begin{Question}\label{Q: characteristic}
Is the linearity of $I_\Delta^{[p]}$ independent of the characteristic of $K$ when $p\geq3$? By Theorem~\ref{Thm: main criterion}, a negative answer requires a graph satisfying \Bp\ whose neighbourhood complex has characteristic-dependent $(2p-1)$-st homology. For $p=2$, can an odd characteristic be exceptional in Proposition~\ref{Prop: exceptional characteristics}? By Remark~\ref{Rem: only 2}, this cannot happen for the graphs $\Delta_M$ of Construction~\ref{Cons: Delta_M}.
\end{Question}

\begin{Question}\label{Q: minimal forbidden graphs}
Is every minimal forbidden graph different from $K_{2,4}$ and $Cr(5)$ of the form $\Delta_M$, for a complex $M$ as in Construction~\ref{Cons: Delta_M}?
\end{Question}

\begin{footnotesize}
{\bf Acknowledgments.}
Naoki Terai was partially supported by JSPS Grant-in Aid for Scientific Research (C) 24K06670. Francesco Navarra and Ayesha Asloob Qureshi were supported by Scientific and Technological Research Council of Turkey T\"UB\.{I}TAK under the Grant No: 124F113. Francesco Navarra is a member of GNSAGA Indam and he acknowledges their support. Experiments with the computer algebra software \textit{Macaulay2} \cite{M2} have provided numerous valuable insights; the authors are grateful to the HPC group at Sabanci University for installing \textit{Macaulay2} on the Tosun cluster and for providing access to such a valuable computational resource. \\

{\bf Declaration of competing interest.}  The authors declare that they have no known competing financial interests or personal relationships that could have appeared to influence the work reported in this paper.\\

{\bf Data availability.} No data was used for the research described in the article.
\end{footnotesize}

\bibliography{ref.bib}{}
\bibliographystyle{alpha}

\end{document}